\documentclass[11pt]{article}
\usepackage[letterpaper]{geometry}
\usepackage[parfill]{parskip}
\usepackage{amsmath,amsthm,amssymb,bbm,bm}
\usepackage{mathtools}
\usepackage{cases}
\usepackage{graphicx}
\usepackage{tabularx}
\usepackage{microtype}
\usepackage{enumitem}
\usepackage{dsfont}

\usepackage{authblk}

\usepackage{url}
\usepackage[colorlinks,citecolor=blue,urlcolor=blue,linkcolor=blue,linktocpage=true]{hyperref}
\usepackage{cleveref}
\crefformat{equation}{(#2#1#3)}
\crefrangeformat{equation}{(#3#1#4) to~(#5#2#6)}
\crefname{equation}{}{}
\Crefname{equation}{}{}

\usepackage[round]{natbib}
\usepackage{multirow}

\newtheoremstyle{mythmstyle}
  {8 pt} % Space above
  {3 pt} % Space below
  {} % Body font
  {} % Indent amount
  {\bfseries} % Theorem head font
  {.} % Punctuation after theorem head
  {.5em} % Space after theorem head
  {} % Theorem head spec (can be left empty, meaning `normal')

\theoremstyle{plain}

\makeatletter
\def\thm@space@setup{%
  \thm@preskip=6pt plus 1pt minus 1pt
  \thm@postskip=\thm@preskip % or whatever, if you don't want them to be equal
}
\makeatother

\newtheorem{theorem}{Theorem}[section]
\newtheorem{lemma}[theorem]{Lemma}
\newtheorem{corollary}[theorem]{Corollary}
\newtheorem{proposition}[theorem]{Proposition}
\newtheorem{remark}{Remark}
\newtheorem{example}{Example}
\newtheorem*{example*}{Example}
\newtheorem{definition}{Definition}
\newtheorem*{definition*}{Definition}
\newtheorem{assumption}{Assumption}
\newtheorem*{remark*}{Remark}
\crefname{definition}{\textbf{definition}}{definitions}
\Crefname{definition}{Definition}{Definitions}
\crefname{assumption}{\textbf{assumption}}{assumptions}
\Crefname{assumption}{Assumption}{Assumptions}

\newcommand{\lstability}{\epsilon_\ell}
\newcommand{\rstability}{\epsilon_{R}}

\DeclareMathOperator*{\argmin}{arg\,min}

\begin{document}
\allowdisplaybreaks
\title{Black-Box Model Confidence Sets Using Cross-Validation with High-Dimensional Gaussian Comparison}

 \author[1]{Nicholas Kissel}
\affil[1]{Carnegie Mellon University}

 \author[1]{Jing Lei}
%\affil[1]{CMU}

\maketitle

\begin{abstract}
  We derive high-dimensional Gaussian comparison results for the standard $V$-fold cross-validated risk estimates.
 Our results combine a recent stability-based argument for the low-dimensional central limit theorem of cross-validation with the  high-dimensional Gaussian comparison framework for sums of independent random variables.  These results give new insights into the joint sampling distribution of cross-validated risks in the context of model comparison and tuning parameter selection, where the number of candidate models and tuning parameters can be larger than the fitting sample size. As a consequence, our results provide theoretical support for a recent methodological development that constructs model confidence sets using cross-validation.
\end{abstract}

\section{Introduction}
%!TEX root = ./hd_cv_clt.tex
Cross-validation \citep{Stone74,Allen74,Geisser75} is among the most popular procedures for estimating the out-of-sample predictive performance of statistical models fitted on data sets randomly sampled from a population. Generally speaking, cross-validation estimates the out-of-sample prediction accuracy by fitting and assessing a fitted model on separate subsets of data. One of the most common forms of cross-validation is $V$-fold cross-validation, where data are partitioned into $V$ folds (sets) of identical size; then, each fold is used to assess the error of the model fitted using the other $V-1$ folds. Finally, the average of all $V$ estimates is used to create the cross-validation risk estimate.

Cross-validation is commonly used in statistical learning problems wherein researchers either compare the cross-validated risk of multiple models or compare a cross-validated risk against some baseline method with known risk. See \cite{picard1984cross,arlot2010survey} for examples.  The popularity and simplicity of cross-validation has inspired  numerous research articles seeking to better understand its theoretical properties. In particular, positive results have been established for parameter estimation following the model selected by cross-validation, including risk consistency and parameter estimation consistency. See \cite{stone1977asymptotics,homrighausen2017risk,chetverikov2016cross,celisse2014optimal} for various examples from linear regression to nonparametric density estimation problems. 

%In this paper we consider the precise uncertainty quantification of cross validated risk estimates.  
Despite the consistency results established for parameter estimation, understanding the model selection properties of cross-validation has been a challenging task, with most existing results being negative.  In the early work of \cite{stone1977asymptotic}, it is shown that cross-validation is similar to AIC, and hence prone to choosing overfitted linear regression models. Such an overfitting tendency of cross-validation has been further studied in \cite{Shao93,Zhang93,yang2007consistency}, which show that in the classical regime, cross-validation often produces inconsistent model selection unless a very unrealistic train-validate ratio is used. Indeed, these ratios are so extreme that they can never be satisfied by standard V-fold cross-validation.  Furthermore, the unsatisfactory model selection performance of cross-validation has been widely observed in practice, and many heuristic or context-specific adjustments have been proposed, such as \cite{efron1997improvements,tibshirani2009bias,yu2014modified}.

The model selection inconsistency of cross-validation can be understood as an instance of the ``winner's curse.'' Since the cross-validated risk of each model is still a random variable, a particular model may have the smallest cross-validated risk because its realized random fluctuation happens to be small while the true optimal model has a much larger fluctuation.  Such an intuition calls for a more precise understanding of the sampling distribution of cross-validated risks.  A main challenge in studying the sampling distribution of cross-validated risk is the global and heterogeneous dependence among each individual empirical loss function.  \cite{bousquet2002stability} proved convergence of cross-validated risk to the corresponding population quantity under an expected leave-one-out loss stability condition.  The population target of cross-validated risk and its variability is further studied in \cite{bates2021cross}.

In this work, we study the simultaneous fluctuations of the cross-validated risks of many models around their mean values. In particular, we establish high-dimensional Gaussian comparison results for the cross-validated risk vector indexed by a collection of models, whose cardinality can potentially be very large.   Our main contributions are two fold. First, we extend the low-dimensional central limit theorem by \cite{austern2020asymptotics} to the high-dimensional case, combining their cross-validation error analysis with the high-dimensional Gaussian comparison framework by \cite{chernozhukov2013gaussian}.  Second, we provide theoretically justifiable model selection confidence sets using cross-validation, answering an open question left in the methodological work \cite{lei2020cross}.
%We will also briefly discuss possible ways to use the Gaussian comparison result to construct confidence sets in cross-validation based risk estimation and model selection.

  % For many years, these kinds of findings did not come equipped with any straight-forward notion of uncertainty. That is, there were not valid inferential procedures for comparing a cross validated risk against some null hypothesis. Establishing conditions for valid inference of cross-validation is admittedly difficult since there's a great deal of dependence between each fitted model. For this reason, the needs of applied researchers have outpaced the ability to create such tools and instead model-class specific procedures were created \cite{something}. Recently however, there has been much work in the way of establishing central limit theorems for the cross-validated risk \citep{austern2020asymptotics, bayle2020cross}. 

Our theoretical development extends and merges two lines of current research: central limit theorems for cross-validation and high-dimensional Gaussian comparisons.  Low-dimensional central limit theorems for cross-validation have been developed recently by \cite{austern2020asymptotics} and \cite{bayle2020cross}.  While these low-dimensional results provide very useful insights for the estimation of prediction risks of individual models, they cannot be used to construct simultaneous confidence sets when many candidate models are being compared.  This is of particular interest because in practice, cross-validation is often used to compare and select from a large collection of models or tuning parameters. Therefore, in order to understand the behavior of cross-validation in selecting from many models, it is necessary to consider the joint sampling distribution of the cross-validated risks.  In the case of sums of independent random vectors, high-dimensional Gaussian comparison has been developed in the milestone work of \cite{chernozhukov2013gaussian} \citep[see also][]{bentkus2005lyapunov}.  Since then, similar results have been developed for $U$-statistics \citep{chen2018gaussian} and stochastic processes with weak dependence such as mixing or spatial process \citep{kurisu2021gaussian,chang2021central}.  However, these extensions do not cover the cross-validation case, where all terms in the summation are dependent on each other and have similar magnitudes of dependence, violating the sparsity of dependence ($U$-statistics) and fast decaying dependence (mixing and spatial processes).  In fact, a different extension of the Gaussian comparison result is needed for cross-validation, which borrows the martingale decomposition and stability conditions in \cite{austern2020asymptotics}.

Stability conditions play a key role in developing the Gaussian comparison results in this work.  Outside of the analysis of cross-validation, the importance of stability has been studied in the broader statistics and machine learning communities \citep{yu2013stability,meinshausen2010stability,bousquet2002stability,hardt2016train}.  To develop our results, we require more subtle versions of stability than in the existing literature, including second order stability, stability in sub-Weibull tails, and stability of difference loss functions.  We provide rigorous justifications of these stability conditions for the stochastic gradient descent algorithm and a prototypical non-parametric regression setting.

% To add to this body of literature, we adjust the work of \cite{austern2020asymptotics} by providing additional tail conditions and making minor corrections. We further extend these results using high-dimensional central limit theorem techniques borrowed from \cite{chernozhukov2015comparison}. This allows us to construct model confidence intervals and uniform confidence bands for cross-validated risk.

\section{Preliminaries}
%!TEX root = ./hd_cv_clt.tex
Consider iid data $\mathbf X = (X_0,X_1,...,X_n)$ with $X_i\in \mathcal X$. We would like to simultaneously study the performance of $p$ learning algorithms through the framework of $V$-fold cross-validation.

For notation simplicity, we assume $V$ evenly divides $n$. 
For $v\in[V]$, let $I_v=\{n(v-1)/V+1,...,n v/V\}$ be the index that corresponds to the $v$th fold of data. Let $\tilde n = n(1-1/V)$ be the training sample size used in $V$-fold cross-validation.
For each $r\in [p]$, let $\ell_r(\cdot;\cdot):\mathcal X\times \mathcal X^{\tilde n}\mapsto \mathbb R$ be a loss function. Intuitively, we should think $\ell_r(x_0;(x_1,...,x_{\tilde n}))$ as the loss function evaluated at $x_0$ of a fitted model using training data $(x_1,...,x_{\tilde n})$. Here the index $r$ denotes a particular model or tuning parameter value. This notation covers both supervised and unsupervised learning.
\begin{enumerate}
 	\item In supervised learning, each point can be thought of as $x=(y,z)\in \mathcal Y\times\mathcal Z$ where $z$ is a vector of covariates and $y$ is the response variable. The loss function can often be written more concretely as 
 	$$\ell_r(x_0;(x_1,...,x_{\tilde n}))=\rho(y_0, \hat f_r(z_0))\,,$$
 	where $\hat f_r(\cdot):\mathcal Z\mapsto \mathcal Y$ is a regression function that predicts $y$ from $z$, trained using the $r$th model/tuning parameter with input data $(x_1,...,x_{\tilde n})$, and
 	 $\rho(\cdot,\cdot):\mathcal Y^2\mapsto\mathbb R$ is a loss function measuring the quality of predicting $y$ using $\hat f_r(z)$, such as squared loss, $0$-$1$ loss, and hinge loss.
 	 \item In unsupervised learning,
 	 $$\ell_r(x_0;(x_1,...,x_{\tilde n}))=\rho(x_0;\hat f_r)\,,$$
 	 where $\hat f_r$ is a function describing the distribution of $X$  trained from the $r$th model/tuning parameter with the input data $(x_1,...,x_{\tilde n})$, and the function $\rho$ is a loss function assessing the agreement of the sample point $x_0$ and the fitted probability model $\hat f_r$.  Examples of $\rho$ include the negative likelihood in density estimation and the proportion of total variance explained in dimension reduction.
 \end{enumerate} 
In model selection and parameter tuning, a particularly interesting scenario is when the number of models being compared is large. 

For each $r$, the $V$-fold cross-validated risk is
\begin{equation}
	\hat R_{{\rm cv}, r} = n^{-1}\sum_{i=1}^n \ell_r(X_i; \mathbf X_{-v_i})\,.
\end{equation}
where $\mathbf X_{-v}$ denotes the sub-vector of $\mathbf X$ excluding the $v$th fold, and $v_i\in[V]$ is such that $i\in I_{v_i}$.

It is natural to expect $\hat R_{{\rm cv}, r}$ to approximate the true average risk of the fitted model:
$$
\tilde R_r=\frac{1}{V}\sum_{v=1}^V R_r(\mathbf X_{-v})\,,
$$
where $$\tilde R_r(\mathbf X_{-v})=\mathbb E\left[\ell_r(X_0;\mathbf X_{-v})|\mathbf X_{-v}\right]\,,$$
is the true risk of the $r$th model fitted using input data $\mathbf X_{-v}$.

The quantity $\tilde R_r(\mathbf X_{-v})$ still depends on the input data $\mathbf X_{-v}$ and hence is a random variable itself.
It would be natural to consider its expected value
$$
R_r^* = \mathbb E \tilde R_r(\mathbf X_{-v}) = \mathbb E\ell_r (X_0;\mathbf X_{-1})\,.
$$

As a statistical inference task, model comparison also involves uncertainty quantification of risk point estimates, and one would hope to establish central limit theorems of the form
$$
\frac{\sqrt{n}(\hat R_{{\rm cv},r}-\mu_r)}{\sigma_r}\rightsquigarrow N(0,1)
$$
with $\mu_r$ being either $\tilde R_r$ or $R_r^*$ and some appropriate scaling $\sigma_r$ \citep{austern2020asymptotics,bayle2020cross}.

In the context of model comparison or tuning parameter selection, these individual normal approximations would have limited practical use.  For example, in our numerical example in \Cref{sec:conf_band}, individual confidence intervals fail to simultaneously cover the targets when when $p$ is moderately large. To cover this gap between theory and practice, we seek to establish a high-dimensional Gaussian approximation in a similar fashion as in \cite{chernozhukov2013gaussian}:
\begin{equation}\label{eq:hd_gaussian_approx}
\sup_{z\in\mathbb R} \left|\mathbb P\left(\max_{1\le r\le p}\sqrt{n}(\hat R_{{\rm cv},r}-\mu_r)\le z\right)-\mathbb P\left(\max_{1\le r\le p} Y_r \le z\right)\right|\rightarrow 0
\end{equation}
for some centered Gaussian random vector $\mathbf Y=(Y_1,...,Y_p)$ with matching covariance.

\section{Main results}
%!TEX root = ./hd_cv_clt.tex
In this section, we establish a high-dimensional Gaussian approximation result with random centering. In particular, we prove \eqref{eq:hd_gaussian_approx} with $\mu_r=\tilde R_r$.  In the following subsections, we present and discuss the assumptions required for this result and provide its full statement as a theorem in \Cref{sec:main_thm}.

\subsection{Symmetry and moment conditions on the loss function}
The idea of cross-validation relies on independence and symmetry among data points. 
We consider the following symmetry and moment conditions on the loss functions $\ell_r$.
\begin{assumption}[Symmetry and moment condition on $\ell_r$]\label{ass:condition_l}
	For each $r\in[p]$, the loss function $\ell_r(\cdot;\cdot)$ satisfies
	\begin{enumerate}
		\item [(a)] $\ell_r(x_0;x_1,...,x_{\tilde n})$ is symmetric in $(x_1,...,x_{\tilde n})$.
%		\item [(b)] $\|\ell_r(X_1;\mathbf X_{-1})\|_4\le 1$.
		\item [(b)] $\mathbb E\left\{[\ell_r(X_1;\mathbf X_{-1})-R_r(\mathbf X_{-1})]^2\right\} \ge \underline{\sigma}^2$ for some constant $0<\underline{\sigma}$. 
	\end{enumerate}
\end{assumption}

Part (b) essentially assumes that the randomly centered cross-validated loss function has non-degenerate conditional variance.  This makes intuitive sense, as we would expect the resulting confidence interval to have length at the scale of $1/\sqrt{n}$.   For example, if $\ell_r$ is a regression residual, then this lower bound is at least as large as the prediction risk of the ideal regression function. In the additional assumptions below, we will also have the upper bound on the variance term. 

\subsection{Stability and tail conditions}
A key consideration from \cite{austern2020asymptotics} in their low dimensional central limit theorems for cross-validation is the stability of the loss function and the average risk when one input sample point is replaced by an iid copy.
In the high dimensional case, we need the loss function to be stable in a uniform sense across all $p$ candidate models indexed by $r\in[p]$. Thus, we will consider stability conditions in the form of stronger tail inequalities instead of the moment conditions used for the low dimensional case.  Such stronger tail conditions are common in high dimensional central limit theorem literature, such as in \cite{chernozhukov2013gaussian}.

We use sub-Weibull concentration to describe the required tail behaviors of random variables.
\begin{definition}[Sub-Weibull Random Variables]
Let $K$ be a positive number, we say a random variable $X$ is $K$-sub-Weibull ($K$-SW for short) if there are positive constants $(a,b,\alpha)$ such that
$$\mathbb P\left(\frac{|X|}{K}\ge t\right)\le a e^{-bt^\alpha}\,,~~\forall~t>0\,.$$
\end{definition}
This definition generalizes the well-known sub-exponential and sub-Gaussian distributions, and has been systematically introduced in \cite{vladimirova2020sub,kuchibhotla2018moving}.  

\begin{remark}\label{rem:sw}
Unlike common practices in the literature, our notation of the sub-Weibull tail inequality  only focuses on the scaling $K$. We do not keep track of the constants $a,b,\alpha$, which can vary from one instance to another as long as they stay bounded and bounded away from zero.  It is easy to check that our notion of sub-Weibull is invariant under constant scaling: if $X$ is $1$-SW, then $X$ is $c$-SW for all positive constant $c$.  Aside from the scaling $K$, the second (and only) important parameter in sub-Weibull tail inequality is the exponent $\alpha$.  In the literature, it is more common to write $(K,\alpha)$-SW.   Our proof can be adapted to keep explicit track of the constant $\alpha$ in each instance at the cost of more complicated bookkeeping, but that does not qualitatively change the results.

In our theoretical developments, the dependence on logarithm terms may be complicated, as it involves the sub-Weibull constant $\alpha$, which may vary between lines.  For brevity of presentation, we absorb such logarithm terms into
 the $\tilde O(\cdot)$ notation. Where $A \le \tilde O(B)$ means that there are positive constants $c_1,c_2$ independent of $(n,p)$ such that 
 $A\le c_1 \log^{c_2}(n+p)B$. 
\end{remark}

To introduce the stability conditions,
let $X_i'$ be iid copies of $X_i$ for $1\le i\le n$ and $\mathbf X^i$ be the vector obtained by replacing $X_i$ with $X_i'$.
For any function $f(\mathbf X)$, define
$\nabla_i f(\mathbf X)=f(\mathbf X)-f(\mathbf X^i)$.

The main stability conditions involved in our normal approximation bounds are the following.
\begin{assumption}\label{ass:lstability}
	There exists $\lstability=n^{-c_\ell}$ for some $c_\ell \in (0,1/2]$ such that
	\begin{enumerate}
		\item [(a)] For all $i\in [\tilde n]$, $r\in[p]$, $\nabla_i\ell_r(X_0,\mathbf X_{-1})$ is $n^{-1/2}\lstability$-sub-Weibull.
		\item [(b)] For all $1\le i<j\le \tilde n$, $r\in[p]$, $\nabla_j\nabla_i\ell_r(X_0,\mathbf X_{-1})$ is $n^{-3/2}\lstability$-sub-Weibull.
		\item [(c)] For all $r$, $\ell_r(X_0,\mathbf X_{-1})$ is $1$-sub-Weibull.
	\end{enumerate}
\end{assumption}
 \Cref{ass:lstability} requires that the first order difference $\nabla_i\ell_r(X_0,\mathbf X_{-1})$ has a scaling no larger than $\lstability n^{-1/2}$, the second order difference $\nabla_j\nabla_i\ell_r(X_0,\mathbf X_{-1})$ has a scaling no larger than $\lstability n^{-3/2}$, and the original loss function $\ell_r$ has a constant scaling.  The sub-Weibull tail ensures that with high probability all such quantities will not exceed their scalings by more than a poly-logarithm factor.  We require $\lstability\in [n^{-1/2},1)$, as this simplifies the presentation of the results and is also the most natural range of stability. Specifically, the approximation error bounds in our main theorems become meaningless if $\lstability>1$, and $\lstability\ll n^{-1/2}$ would be impractical, as it implies changing one sample point will incur a change less than $1/n$ in the loss.

We further remark that the scaling assumption on the second order difference \linebreak$\nabla_j\nabla_i\ell_r(X_0,\mathbf X_{-1})$ is stronger than that in \cite{austern2020asymptotics} by a factor of $\sqrt{n}$.  This is due to a fundamental difference between the low dimensional CLT and high dimensional Gaussian comparison, where the former only requires controlling the second moment of error terms, while the latter requires controlling the supremum of many such error terms.
More specifically, define the randomly centered loss at $X_i$
\begin{equation}
K_{r,i} =  \ell_r(X_i;\mathbf X_{-v_i}) - R_r(\mathbf X_{-v_i})\,,\label{eq:K}
\end{equation}
%J_{r,i} = & \ell(X_i;\mathbf X_{-v_i}) - \ell(X_i';\mathbf X_{-v_i})=\nabla_i \ell(X_i;\mathbf X_{-v_i})\,,\\
and
\begin{equation}
D_{r,i} =  \sum_{j\notin I_{v_i}} \nabla_i K_{r,i}\,.\label{eq:D}
%L_{r,i} = & \sum_{j\notin I_{v_i}} \nabla_i R(\mathbf X_{-v_j})\,.
\end{equation}
A key result in the low dimensional CLT  is that
$\|D_{r,i}\|_2\lesssim \lstability$ provided $\|n^{1/2}\nabla_i\ell_r(X_0,\mathbf X_{-1})\|_2 \allowbreak\le \lstability$
and $n\|\nabla_j\nabla_i\ell_r(X_0,\mathbf X_{-1})\|_2\le \lstability$.  However, in  the high-dimensional regime, we need to 
simultaneously control $D_{r,i}$ for all $1\le r\le p$, which cannot be guaranteed by a vanishing second moment on 
each individual term.  Our condition can be relaxed to requiring a similar $\nabla_j\nabla_i\ell_r(X_0,\mathbf X_{-1})$ being $n^{-1}\lstability$-sub-Weibull, provided we can further assume that $\frac{D_{r,i}}{\|D_{r,i}\|_2}$ is $1$-sub-Weibull.  While this additional assumption certainly seems reasonable in many situations, we choose to work with the stronger condition on the second order difference as presented in \Cref{ass:lstability}, which allows for a more streamlined presentation.
Nevertheless, the stability conditions in \Cref{ass:lstability} are still practically plausible since we should typically expect each $\nabla$ operator to reduce the scale of the loss function by a factor up to $n$.

 In \Cref{sec:stability}, we provide two rigorous examples that satisfy the stability conditions, including stochastic gradient descent with convex and smooth objective functions, and a non-parametric regression with sub-Weibull design.

%Here we describe two examples with simple intuition about the plausibility of the stability conditions required in \Cref{ass:lstability}. %  Rigorous derivations for these  seem possible by invoking convexity and smoothness in M-estimation (Example 1, see also Proposition 4 of \cite{austern2020asymptotics}), and random matrix theory (Example 2).   Another example satisfying such stability conditions is bagged estimates \citep{chen2022debiased}. 

\begin{example}[Classical M-Estimator]
	Consider a parametric loss function $\ell(x_0;x_1,...,x_{\tilde n})=\ell(x_0;\hat\theta)$ with $\hat\theta$ estimated from the input data $(x_1,...,x_{\tilde n})$.  Under classical parametric regularity conditions such $\hat\theta$ can 
	%satisfy the stability required in \Cref{ass:lstability}. For example, if $\hat\theta$ is 
	 be asymptotically linear \cite[][Chapter 3]{tsiatis2006semiparametric}. Then we have
	$$
    \hat\theta =\theta_0 +\frac{1}{\tilde n}\sum_{i=1}^{\tilde n}\varphi(x_i)+o_P(n^{-1/2})\,.
	$$
	Intuitively, the first order stability bound for $\nabla_i \ell$ can be satisfied if
	the $o_P(n^{-1/2})$ remainder term is $n^{-1/2}\epsilon_\ell$-SW for some $\epsilon_\ell=o(1)$.
	The second order stability condition would require more subtle structure within the $o_P(n^{-1/2})$ remainder term.  \cite{austern2020asymptotics} gave a formal analysis of the first and second order statbility conditions for $M$-estimators under convexity and smoothness.
\end{example}

\begin{example}[Penalized Least Squares]
	Now consider a high dimensional ridge-regression where we have paired sample points $x_i=(z_i,y_i)$:
	$$\hat\beta=\arg\min {\tilde n}^{-1}\sum_{i=1}^{\tilde n}(y_i-z_i^T\beta)^2+\lambda\|\beta\|_2^2\,.$$
When the dimensionality of $x_i$ is comparable or smaller than the sample size, it is possible to argue that the empirical covariance matrix will be well-conditioned with high probability, and hence changing any one sample point will incur an $O(n^{-1})$ change in $\hat\beta$.  
For a simpler argument under stronger assumptions, if the sample points are bounded and $\lambda\gg n^{-1/2}$, then the stability requirement on $\nabla_i\ell$ holds.  If $\lambda\gg n^{-3/4}$ then the stability requirement on $\nabla_j\nabla_i\ell$ also holds.
\end{example}

\subsection{Main theorem with random centering}\label{sec:main_thm}
Our first main result is a Gaussian comparison with random centering.  In order to state the result, we need to specify the covariance of the Gaussian vector $\mathbf Y$. Using the notation $K_{r,0,v}=\ell_r(X_0;\mathbf X_{-v})-R_r(\mathbf X_{-v})$, define
% $$
% \sigma_{r,v}^2=\mathbb E(K_{r,0,v}^2|\mathbf X_{-v})
% $$
% and
$$
\sigma_{rs,v}=\mathbb E (K_{r,0,v}K_{s,0,v}|\mathbf X_{-v})\,,~~1\le r,s\le p\,,
$$
to be the conditional variance/covariance of the loss functions given the fitted model using input data $\mathbf X_{-v}$,
and let
$$
\sigma_{rs} = \mathbb E\sigma_{rs,1}
$$
be the expected value of $\sigma_{rs,v}$.
Let $\Sigma=[\sigma_{rs}]_{1\le r,s\le p}$ be the corresponding $p\times p$ expected conditional covariance matrix, and $\Sigma_v$ the corresponding random covariance matrix with entries $\sigma_{rs,v}$.

We will show that $\sigma_{rs,v}\approx \sigma_{rs}$ (\Cref{lem:cov_concentration}) and $\sqrt{n}(\hat{\mathbf R}_{\rm cv}-\tilde{\mathbf R})$ behaves like a centered Gaussian vector with covariance matrix $\Sigma$.   
% Consider the sequence of conditional Gaussian vectors 
% $$
% \mathbf Y_{i}= \Sigma^{1/2}\boldsymbol{\varepsilon}_{i}\,,~~1\le i\le n\,,
% $$
% where $\boldsymbol\varepsilon_{i}\stackrel{iid}{\sim}N(0,I_p)$ and are independent of everything else. 

\begin{theorem}[High-dimensional CLT for Cross-validation with random centering]\label{thm:1}
Assume \Cref{ass:condition_l,ass:lstability} hold,
then we have
	\begin{equation*}
	\sup_{z\in\mathbb R} \left|\mathbb P\left(\max_{1\le r\le p}\sqrt{n}(\hat R_{{\rm cv},r}-\tilde R_r)\le z\right)-\mathbb P\left(\max_{1\le r\le p} Y_i \le z\right)\right|\le \tilde O\left[n^{-1/8}+\lstability^{1/3}\right]\,,
	\end{equation*}
	for $\mathbf Y=(Y_1,...,Y_p)\sim N(0,\Sigma)$.
\end{theorem}

\begin{remark}
\Cref{thm:1} implies that the Gaussian approximation error is small if $\lstability\lesssim n^{-c}$ for some constant $c>0$.  The result of \Cref{thm:1} can be easily extended to the quantity $\max_r |\hat R_{{\rm cv},r}-\tilde R_{r}|$ by applying \Cref{thm:1} to the augmented vector $(\hat{\mathbf R}_{\rm cv}-\tilde{ \mathbf R}, \tilde {\mathbf R}-\hat{\mathbf R}_{\rm cv})$ with the corresponding Gaussian vector $(\mathbf Y, -\mathbf Y)$.
\end{remark}

\begin{remark}\label{rem:constant_sigma}
In addition to a factor of 
$\log^c(p+n)$ with constant $c$ determined by the sub-Weibull exponents in \Cref{ass:lstability}, the $\tilde O(\cdot)$ notation in \Cref{thm:1} also contains a factor that depends on $\underline{\sigma}$, the lower bound of conditional risk standard deviation in \Cref{ass:condition_l}.  We suppressed this factor throughout this paper because a Gaussian comparison is practically most useful when the matching Gaussian process $Y$ is marginally standardized, so that $\underline\sigma=1$.  Otherwise, the maximum will be largely driven by the coordinates with large variances.  Such factors involving $\underline\sigma$ have been treated as constants in the literature of high-dimensional Gaussian comparison and anti-concentration of maxima of Gaussian processes \citep{chernozhukov2013gaussian,chernozhukov2015comparison}.   With more detailed bookkeeping in our proof, and inspecting the proof of the anti-concentration results in \cite{chernozhukov2015comparison}, the contribution of $\underline\sigma$ in the error bound in \Cref{thm:1} is a multiplicative factor of $O(\underline\sigma^{-2})$.
%can be written with an explicit dependence on $\underline{\sigma}$:
% \begin{align*}
% 	&\sup_{z\in\mathbb R} \left|\mathbb P\left(\max_{1\le r\le p}\sqrt{n}(\hat R_{{\rm cv},r}-\tilde R_r)\le z\right)-\mathbb P\left(\max_{1\le r\le p} Y_i \le z\right)\right|\\
% 	\le & c_1 \log^{c_2}(n+p)\left[n^{-1/8}+\underline{\sigma}^{-2}\lstability^{1/3}\log^{1/3}(\lstability^{-1})\right]\,,
% 	\end{align*}
\end{remark}
% for constants $c_1$, $c_2$ independent of $(n,\lstability,\underline{\sigma})$.

\section{Simultaneous confidence bands for cross-validated risk}\label{sec:simu_band}
%!TEX root = ./hd_cv_clt.tex

% In practice, inference based on Gaussian approximations with random centering and scaling involves calibrating the randomness of the conditional variances $$\sigma_{r,v}^2=\sigma_r^2(\mathbf X_{-v})={\rm Var}\left[\ell(X_0; \mathbf X_{-v})|\mathbf X_{-v}\right]\,,~~r\in [p]$$
% and the conditional covariances 
% $$
% \sigma_{rs,v} = \sigma_{rs,v}(\mathbf X_{-v})={\rm Cov}\left[\ell_r(X_0; \mathbf X_{-v}),~\ell_s(X_0; \mathbf X_{-v})\big|\mathbf X_{-v}\right]\,,~~1\le r<s\le p\,.
% $$
% In this section we study this problem in a simple but important scenario, where the conditional variances concentrates around its expectation.  The more challenging approach would be considering a deterministic centering and scaling in the Gaussian comparison.  This is pursued in \Cref{sec:thm3} below.

In this section we consider various statistical inference tools following from \Cref{thm:1}, including  constructing simultaneous confidence bands of the average fitted risks and possible ways to construct confidence sets of the ``optimal'' model.

\subsection{Confidence bands}\label{sec:conf_band}
% \Cref{lem:cov_concentration} suggests that under the stability assumption, $\sup_{r,s,v}|\sigma_{rs,v}-\sigma_{rs}|\le \tilde O(\lstability)$ with high probability.  As a consequence, inspite the random vector $\mathbf Y=n^{-1/2}\sum_i \Sigma_v^{1/2} \boldsymbol{\varepsilon}_i$ being a Gaussian mixture, it behaves much like a single fixed Gaussian vector with covaraince $\Sigma$.  
Following \Cref{thm:1}, we consider the coordinate-wise standardized process
$$\sqrt{n}\Lambda^{-1/2}(\hat{\mathbf R}_{\rm cv}-\tilde{\mathbf R})\,,$$
where
\begin{align*}
\Lambda = &{\rm diag}(\sigma_{11},...,\sigma_{pp})
\end{align*}
is the diagonal submatrix of $\Sigma$.
% \begin{align*}
% \Theta = &\mathbb E[\mathbf K_1\mathbf K_1^T|\mathbf X_{-1}]\,,
% \end{align*}
% with $\mathbf K_i=(K_{1,i},...,K_{p,i})^T$.

Let $\hat\Lambda$ and $\hat\Sigma$ be the natural plug-in estimates (i.e., the average of all the within-fold empirical covariance matrices) of $\Lambda$ and $\Sigma$. In particular, let $\boldsymbol{\ell}=(\ell_1,...,\ell_p)$, and $\hat\Sigma_v$ be the empirical covariance of $\left\{\boldsymbol{\ell}(X_i;\mathbf X_{-v}):i\in I_v\right\}$. Then $\hat\Lambda$ and $\hat\Sigma$ can be the aggregated estimate.
\begin{align}\label{eq:global_var_est}
	\hat\Sigma = &\frac{1}{V}\sum_{v=1}^V\hat\Sigma_v\,,~~~
	\hat\Lambda = {\rm diag}(\hat\sigma_{11},\dots,\hat\sigma_{pp})\,.
\end{align}

Given a nominal type I error level $\alpha\in(0,1)$, the following fully data-driven procedure computes a simultaneous normalized confidence band for the vector $\tilde{\mathbf R}$ with asymptotic coverage $1-\alpha$.

 Let $\hat z_\alpha$ be the upper $1-\alpha$ quantile of $\|Z\|_\infty$, with $Z\sim N(0,\hat\Lambda^{-1/2}\hat\Sigma\hat\Lambda^{-1/2})$.  Given estimates $(\hat\Lambda,\hat\Sigma)$, $\hat z_\alpha$ can be approximated efficiently using Monte-Carlo methods. Such a Monte-Carlo approximation error can be controlled by combining the standard Dvoretsky-Kiefer-Wolfowitz inequality and anti-concentration of Gaussian maxima. For the sake of brevity, we use the theoretical value $\hat z_\alpha$, which corresponds to the limiting case of infinite Monte-Carlo sample size.

 The simultaneous confidence band for cross-validated risk is
	\begin{align}
	\widehat{\rm CI}_r= \left[\hat R_{{\rm cv},r}-\frac{\hat\sigma_{rr}^{1/2} \hat z_\alpha}{\sqrt{n}}\,,~\hat R_{{\rm cv},r}+\frac{\hat\sigma_{rr}^{1/2} \hat z_\alpha}{\sqrt{n}}\right]\,,~~\text{for each } r\in[p]\,,\label{eq:ci}\end{align} 
where $\hat\sigma_{rr}$ is the $r$th diagonal entry of $\hat\Sigma$ in \eqref{eq:global_var_est}.

\begin{corollary}\label{cor:simul_ci}
	Under Assumptions \ref{ass:condition_l} and \ref{ass:lstability}, the confidence intervals constructed in \eqref{eq:ci} satisfy
	$$
    \mathbb P\left(\tilde R_{r}\in \widehat{\rm CI}_r\,,~\forall~r\in[p]\right)\ge 1-\alpha-\tilde O(n^{-1/8}+\lstability^{1/3})\,.
	$$
\end{corollary}

\subsection{Model confidence set}
Now, we consider the model/tuning selection problem.
Let $$r^*=\arg\min_{r}\tilde R_{r}$$ be the index of the candidate model with the smallest average fitted risk.  We hope to use the Gaussian comparison to construct a confidence set for $r^*$.  A simple way to do so is directly using \eqref{cor:simul_ci}:
\begin{equation}\label{eq:naive_conf_set}
\hat{\mathcal A}_0=\left\{r: \hat R_{{\rm cv},r}-\hat\sigma_{rr}^{1/2}\hat z_\alpha/\sqrt{n}\le \min_{s} \hat R_{{\rm cv},s}+\hat\sigma_{ss}^{1/2}\hat z_\alpha/\sqrt{n}\right\}\,.
\end{equation}
It is a direct consequence of \Cref{cor:simul_ci} that 
$$
\mathbb P(r^*\in\hat{\mathcal A}_0)\ge 1-\alpha-\tilde O(n^{-1/8}+\lstability^{1/3})\,.
$$
However, $\hat{\mathcal A}_0$ often unnecessarily contains too many models, as it ignores the correlations among the coordinates of $\hat{\mathbf R}_{\rm cv}$.

Following the idea in \cite{lei2020cross}, we instead consider the following difference based method, which takes into account the correlations of the cross-validated risks.
For each $r$, consider the risk difference vector 
$(\hat R_{{\rm cv},r}-\hat R_{{\rm cv},s}:s\neq r)$, and apply the above framework to
$\ell^{(r)}_s\coloneqq \ell_{r}-\ell_{s}$ to test whether $\hat R_{{\rm cv},r}-\hat R_{{\rm cv},s}>0$ for some $s\neq r$.
Here, we are considering a one-sided hypothesis, so 
instead of the two-sided confidence band in \Cref{eq:ci}, we consider the one-sided version.

For each $r\in[p]$, consider $p-1$ difference loss functions $\ell^{(r)}_s=\ell_r-\ell_s$ for $s\in[p]\backslash\{r\}$. Now apply the cross-validation normal approximation theory to the $p-1$ standardized loss functions $(\ell^{(r)}_s/\sigma_{ss}^{(r)}:1\le s\le p,~s\neq r)$, where
$$
\sigma^{(r)}_{st} = \mathbb E\left\{{\rm Cov}\left[\ell^{(r)}_s(X_0;\mathbf X_{-1}),\ell^{(r)}_t(X_0;\mathbf X_{-1})\big|\mathbf X_{-1}\right]\right\}\,.
$$ Let $\hat z^{(r)}_{\alpha}$ be the $1-\alpha$ quantile of the maximum of the corresponding $(p-1)$-dimensional Gaussian vector with estimated covariance $[\hat\Lambda^{(r)}]^{-1/2}\hat\Sigma^{(r)}[\hat\Lambda^{(r)}]^{-1/2}$, where $\hat\Sigma^{(r)}=[\hat\sigma^{(r)}_{st}]_{s,t\in[p]\backslash\{r\}}$ and $\hat\Lambda^{(r)}$ is its diagonal version. % counterparts of $\hat\Sigma$ and $\hat\Lambda$ when applied to the $(p-1)$ difference loss functions $(\ell_r-\ell_s:s\in[p]\backslash\{r\})$.  
Then our model confidence set is
% Let $\widehat{\rm CI}^{(r)}_s$ be the corresponding confidence interval for $\bar R_{{\rm cv},r}-\bar R_{{\rm cv},s}$ contructed from the difference loss functions $(\ell^{(r)}_s:s\neq r)$.
% The difference-based model confidence set is
\begin{equation}\label{eq:diff_conf_set}
\hat{\mathcal A} = \left\{r: \sup_{s\neq r}\hat R_{{\rm cv},r}-\hat R_{{\rm cv},s}-[\hat\sigma^{(r)}_{ss}]^{1/2}\hat z^{(r)}_\alpha/\sqrt{n} \le 0\right\}\,.
\end{equation}
\begin{proposition}\label{pro:cvc_coverage}
	If the conditionally standardized difference loss functions 
	$\frac{\ell^{(r)}_s-(\tilde R_r-\tilde R_s)}{\sqrt{\sigma^{(r)}_{ss}}}$ satisfy \Cref{ass:condition_l,ass:lstability} for all $r,s$, then
	$$
\mathbb P\left(r^*\in\hat{\mathcal A}\right)\ge 1-\alpha-\tilde O(n^{-1/8}+\lstability^{1/3})\,.
	$$
\end{proposition}

\begin{remark}\label{rem:lossdiff}
\Cref{pro:cvc_coverage} resolves an outstanding question about the theoretical justification of the V-fold cross-validation with confidence method \citep[CVC][]{lei2020cross}. The proof is almost identical to that of Therem 1 in \cite{lei2020cross} except using the cross-validation Gaussian approximation (\Cref{thm:1}), and hence is omitted. Requiring the standardized difference loss functions to satisfy \Cref{ass:lstability} is non-trivial, because if the loss functions $\ell_r$ and $\ell_s$ are highly correlated, their difference can have very small variance.  
Intuitively, if $\nabla_i \ell_r\asymp n^{-1}$, for the stability condition to hold for the standardized difference loss we will need $\|\ell_r-\ell_s\|_2\gg n^{-1/2}$.   
%One can modify the proof of \Cref{thm:1} for the difference loss if the variance of $\ell_r-\ell_s$ vanishes sufficiently slowly.  
Such a slow vanishing requirement on $\ell_r-\ell_s$ precludes the case that both model $r$ and model $s$ produce $\sqrt{n}$-consistent estimates.  This intuition agrees with \cite{yang2007consistency}, which suggests that cross-validation may not be model selection consistent if both candidate models are $\sqrt{n}$-consistent.  A simple illustration of this issue is given in Section 2.3 of \cite{lei2020cross}.  In \Cref{sec:stability}, we give a concrete nonparametric regression example in which the stability conditions hold for the difference losses. 
\end{remark}

\subsection{Numerical Experiments}
In this subsection, we numerically verify the claim of \Cref{thm:1} as well as the model confidence sets considered in \Cref{sec:conf_band}.  We use $V=5$ in all simulations and all plotted values are averaged over 1000 generated data sets.

\paragraph{Simultaneous coverage vs marginal coverage.}
We first investigate the simultaneous coverage of confidence bands for the cross-validated risk. To do so, we generate a predictor matrix $Z \sim N(0, I)$ and response vector $Y = Z \beta + \epsilon$  where $\epsilon \sim N(0,  \sigma^2  I)$ with $\sigma^2=\|\beta\|_2^2/\nu$ and $\beta$ is a sparse $d$-dimensional vector with the first $s$ entries being 1 and the remainder being 0.  
%Also, $\nu$ is specified to enforce a desired signal-to-noise ratio. 
We set $d=20$, $s=5$, and $\nu=1000$.   We then fit lasso regressions across a grid of 50 regularization parameters and generate confidence bands. 

\begin{figure}[]\centering
\includegraphics[width=1\columnwidth]{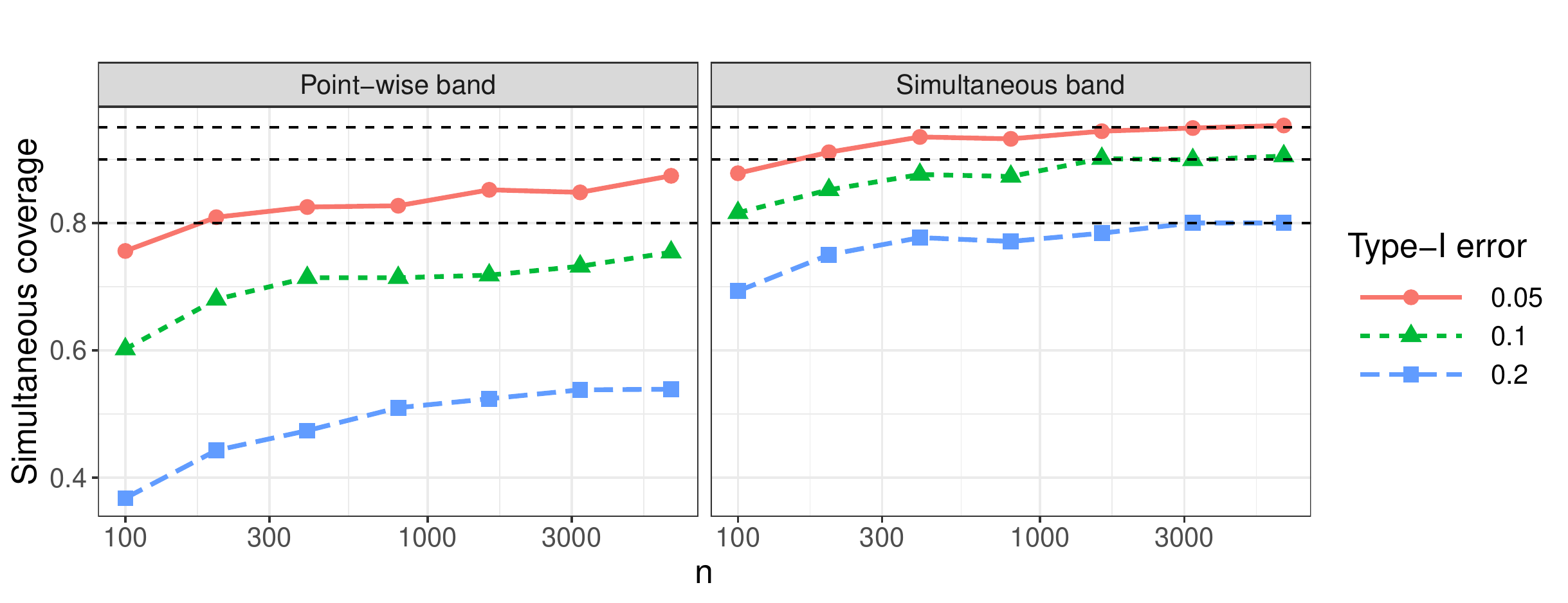}
\caption{\label{fig:coverbandALL_s5_p20_rh0_SNR1000}  Simultaneous coverage of point-wise confidence intervals and the confidence band from \Cref{eq:ci}. The horizontal dashed lines represent the nominal level given by $1-\alpha$. }
\end{figure}

Figure \ref{fig:coverbandALL_s5_p20_rh0_SNR1000} shows the simultaneous coverage of a confidence band generated by all point-wise confidence intervals (left) and the confidence band as specified in \Cref{eq:ci} at various values of $n$ and $\alpha$. We see that the latter method has coverage much closer to the nominal level than the former. Therefore, the point-wise procedure is insufficient for providing the correct simultaneous coverage, suggesting that the simultaneous adjustment is indeed necessary.  This simulation also suggests that the coverage of the simultaneous band is not overly conservative.

\begin{figure}[t]\centering
\includegraphics[width=1\columnwidth]{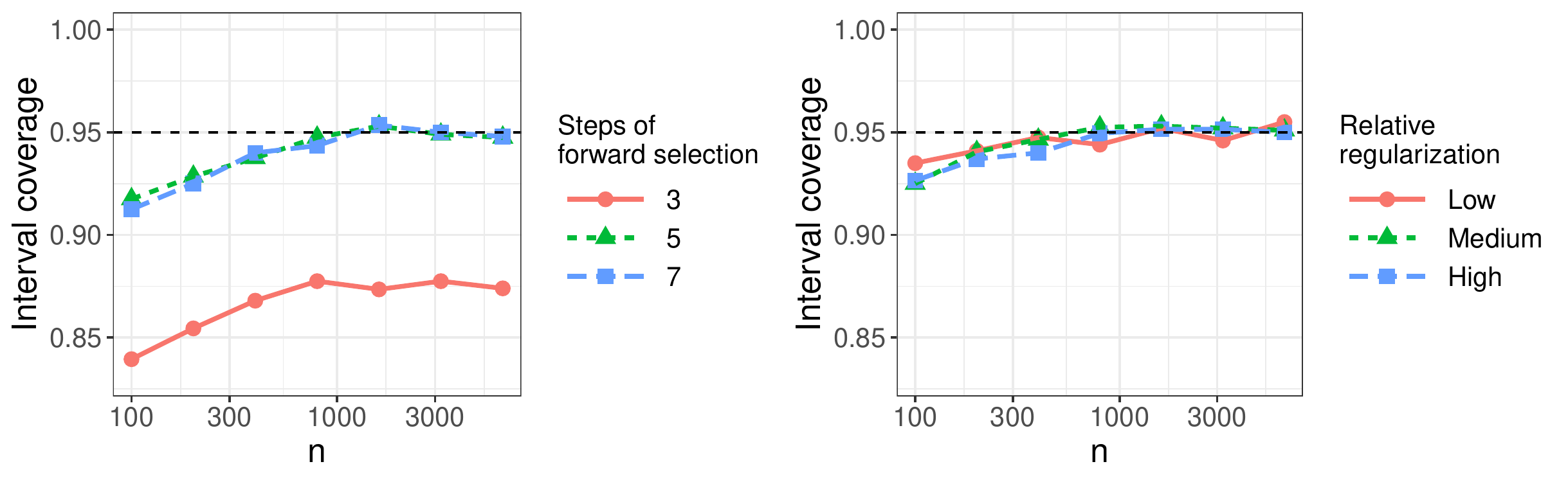}
\caption{\label{fig:fwd_all_summary_b1_s5rh0} Coverage of CV risk confidence intervals for forward selection terminated at different values (left) and lasso at different regularization values (right). The dashed horizontal line marks the nominal level given by $1-\alpha$.}
\end{figure}

\paragraph{The importance of stability.}
The impact of stability on the quality of Gaussian approximation of cross-validated risks has been experimented in \cite{austern2020asymptotics}. Here we provide additional empirical results on this front.  Consider confidence intervals for forward selection in the same setting but with $d=10$ and $\alpha=0.05$. Specifically, we look at the point-wise interval coverage of forward selection terminated at different model sizes--one less than $s$, one equal to $s$, and one larger than $s$. The left plot in \Cref{fig:fwd_all_summary_b1_s5rh0} shows that at 3 steps, the one-dimensional $95\%$-confidence interval based on \Cref{thm:1} under covers regardless of the sample size, while at 5 and 7 steps the coverage converges to nominal level as the sample size increases. This observation is consistent with the stability condition.  Before $s$ is reached, forward selection is quite unstable in this setting, as the non-zero entries of $\beta$ have the same magnitude. Therefore, the algorithm is equally likely to pick any subset of $s-1$ non-zero coordinates, and changing the value of one data point can incur a change of selected variables with non-negligible chance, resulting in instability of the loss function.  When forward selection reaches exactly $s$ steps, it selects the correct subset with overwhelming probability, leading to a stable loss function.  When forward selection selects one more variable, the index of the additional selected variable is not stable but the fitted coefficient is very close to zero, so that the fitted loss function is still stable.  In contrast, the lasso algorithm is continuous in the input data, hence the stability is much easier to hold for all values of penalty parameters, as suggested by the right plot in \Cref{fig:fwd_all_summary_b1_s5rh0}.

\paragraph{The advantage of difference-based model confidence sets.}
Now we look at the model confidence set performance as applied to the lasso on data with increasing $n$. Specifically, we are studying the size $|\mathcal A|$ and coverage $\mathbb P(r^*\in{\mathcal A})$ for
\begin{enumerate}
 	\item [(1)] the na\"{i}ve method $\mathcal{A} = \hat{\mathcal A}_0$ as defined in \eqref{eq:naive_conf_set}, and
 	\item [(2)] the difference based method  $\mathcal{A} = \hat{\mathcal A}$ as defined in \eqref{eq:diff_conf_set}.
 \end{enumerate} This simulation setting is again similar with the value of $s$ fixed at 5, but $d$ grows at rate $n/10$. This time, to make our grid of regularization parameters, we first find $\lambda_{\text{max}} = \frac{1}{n} \Vert Z^T Y \Vert_\infty$, then the grid is defined by $\lambda_{\text{max}}  2^{i}/ \sqrt{1-1/V}$ for $i \in \{0, \dots, -9 \} $. The re-scaling of $\sqrt{1-1/V}$ is done since the choice of $\lambda$ in lasso is inversely proportional to the square root of the training sample size, and $\lambda_{\max}$ may be a bit too small for the reduced sample size in $V$-fold cross-validation. The left plot of \Cref{fig:setcoverSIZE_SNR1_r10_o10} shows that the difference based method produces considerably smaller sets while maintaining coverage of at least 0.95, supporting the intuition that the difference based method is able to take into account the joint randomness of the cross-validated risks.  On the right plot of \Cref{fig:setcoverSIZE_SNR1_r10_o10}, the empirical coverage of the na\"{i}ve method is always overly conservative, while the coverage of the difference based method does get closer to nominal level for certain values of $n$.  Intuitively, such a fluctuation of coverage as $n$ varies can be explained by whether the problem is close to the boundary of the null hypothesis.  More specifically, let $\boldsymbol{\delta}_r=(\tilde R_{r}-\tilde R_{s}:s\neq r)$.  In the difference based method, the null hypothesis for a candidate $r$ is that $\boldsymbol{\delta}_r$ is non-positive in each coordinate, which is a composite null hypothesis. The Gaussian approximation is derived precisely for the extreme point of the null hypothesis, where all coordinates of $\boldsymbol{\delta}_r$ are zero.  In practice, we will never really be working in this scenario.  Therefore, the supremum-based confidence set will be too conservative if $\boldsymbol{\delta}_r$ has many large negative coordinates, and will be nearly exact if most of the coordinates are close to $0$.   In our simulation, when $n$ increases, the relative performance of different tuning parameters also changes.  Indeed we do observe that when $n\le 2500$, the small sample size cannot quite distinguish the two best $\lambda$ values that perform nearly equally well.
 
 \begin{figure}[t]\centering
\includegraphics[width=1\columnwidth]{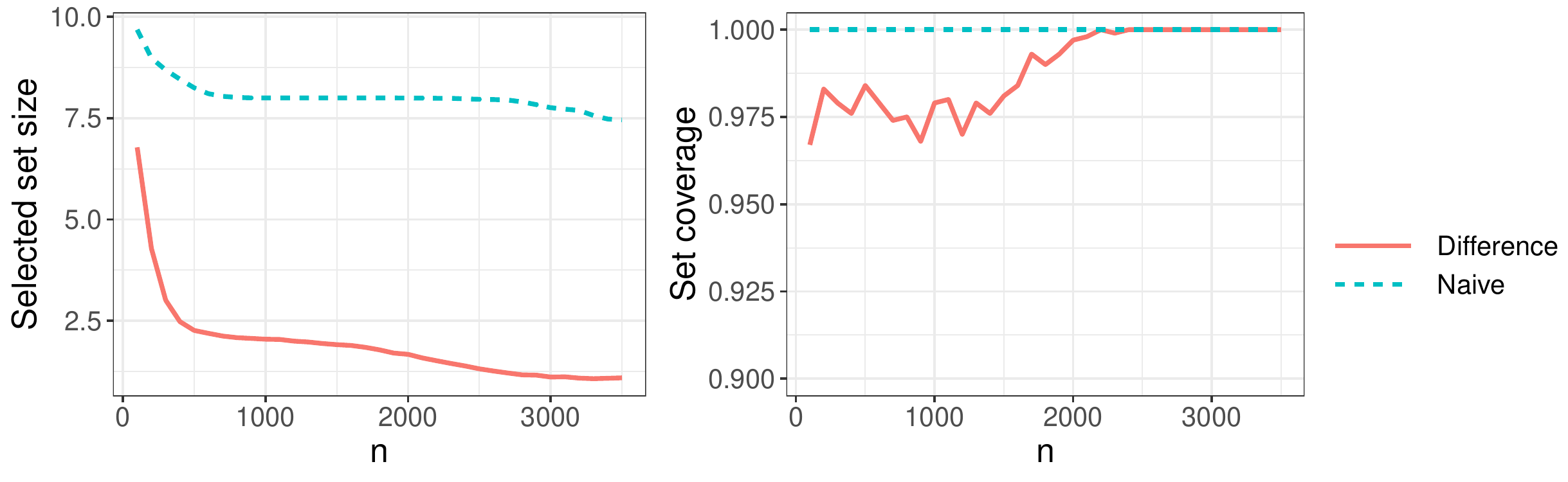}
\caption{\label{fig:setcoverSIZE_SNR1_r10_o10} Size and coverage of model confidence set procedures with $\alpha = 0.05$. }
\end{figure}

\section{Stability condition examples}\label{sec:stability}
%!TEX root = ./hd_cv_clt.tex
Before moving on to our Gaussian comparison results with deterministic centering, we pause to provide some context on the aforementioned stability conditions and  discuss some settings wherein they are satisfied.
Previously, the first order stability condition and its use in proving risk consistency of cross-validation were studied in \cite{bousquet2002stability}.  In the study of low-dimensional CLTs for cross-validation, \cite{austern2020asymptotics} studied both first and second order stability conditions for parameters estimated as global minimums of convex optimization problems.  In this section, we combine and extend these two lines of results.  First, we prove the first and second order stability for stochastic gradient descent applied to convex and smooth objective functions.  This result provides a more concrete example of stability by taking into consideration the particular optimization algorithm and its optimization error.  We also believe our result has the correct technical conditions on the objective function for the second order, which is missing in \cite{austern2020asymptotics}.  Second, we provide an example in which the first and second order stability conditions hold for the difference loss function. This is a much more subtle task since the differences in loss functions often have vanishing variances.

\subsection{Stability of stochastic gradient descent}

Stochastic gradient descent (SGD) is one of the most popular optimization algorithms in large-scale machine learning. 
Thus stability of SGD is particularly important in our analysis of cross-validation. First order stability bounds for SGD were established in \citet{hardt2016train} under convexity and smoothness assumptions, similar to those presented here and in \citet{austern2020asymptotics}. Here, we will provide first and second order stability bounds for SGD.

Suppose the loss function $\ell:\mathcal X \rightarrow \mathbb R$ is written as $\ell(x) = \ell(x;\theta)$ and parameterized by $\theta \in \Theta$. A classical approach for estimating $\theta$ is to use an M-estimator,
\begin{equation}\label{eq:M-est}
    \hat \theta = \argmin_{\theta \in \Theta^\prime \subset \Theta} \sum_{i=1}^n \psi(\theta;X_i),
\end{equation}
where $\psi:\Theta \rightarrow \mathbb R$ is an objective function, which may or may not be the same as $\ell$. A popular way to solve \eqref{eq:M-est} is to use SGD, whereby we initialize $\hat \theta_0 = 0$ and then iteratively update
\begin{equation}
    \hat \theta_{t+1} = \hat \theta_t - \alpha_{t+1} \dot \psi(\hat \theta_t;X_{t+1})
\end{equation}
for $t = 1,\dots,n-1$ with step size $\alpha_t \geq 0$, where  $\dot \psi(\theta; x)$ and $\ddot \psi(\theta; x)$ are the first and second partial derivatives of the objective function $\psi(\cdot,\cdot)$ with respect to $\theta$, respectively. Upon reaching $t=n-1$, we set $\hat \theta = \hat \theta_n$ which serves as our SGD estimator.

%If the loss function $\ell(x;\theta)$ is Lipschitz in 

In order to obtain a first order stability bound for SGD, we consider the following smoothness and convexity assumptions of the objective function. 
\begin{assumption}[Smoothness and convexity of the objective function]\label{asm:sgd1}
For all $x$,
\begin{enumerate}
    \item[(a)] $\psi(\cdot;x)$ is $\gamma$-strongly convex and L-Lipschitz 
    \item[(b)] $\dot\psi(\cdot;x)$ is $\beta$-Lipschitz
    %\item[(c)] $\psi(\cdot;x)$ is $\gamma$-strongly convex 
\end{enumerate}
\end{assumption}
These  assumptions are standard in the learning-theoretic literature, and have been considered in \cite{hardt2016train} to establish a similar first order stability result. 

We now present our first order stability bound for SGD.
\begin{proposition}[First order stability of SGD]\label{prop:sgd1}
    Under \Cref{asm:sgd1}, if $\alpha_t = t^{-a}\beta^{-1}$ for some $a \in (0,1)$ and $\frac\gamma\beta\geq\frac{a(1-a)}{1-2^{-(1-a)}}\frac{\log n}{n^{1-a}}$, then
    \[
        \left\Vert \nabla_i \hat\theta_n \right\Vert \leq \frac{2L}{\beta}n^{-a}
    \]
\end{proposition}
%\Cref{asm:sgd1} has been considered in \cite{hardt2016train} to establish a similar first order stability result for SGD.  \Cref{prop:sgd1} differs from the results in \cite{hardt2016train} by considering a different decay speed of the step size.  \cite{austern2020asymptotics} used the same assumptions to establish the first order stability for global optimum of M-estimators, but did not consider any optimization errors.
This proposition differs from the results in \cite{hardt2016train} by considering a different decay speed of the step size.  Also, \cite{austern2020asymptotics} used the same assumptions to establish the first order stability for the global optimum of M-estimators, but did not consider any optimization errors. 

Before moving on to second order stability, we  provide an example in which \Cref{asm:sgd1} is satisfied.
\begin{example}
Consider ridge regression: $x=(y,z)\in \mathbb B(R_x)$, and $\theta\in \mathbb B(R_\theta)$, where $\mathbb B(R)$ denotes a centered Euclidean ball with radius $R$. Consider objective function
$$
\psi(\theta;x) = \frac{1}{2}(y-z^T\theta)^2+\frac{1}{2}\lambda \|\theta\|^2\,.
$$
Then $\psi$ is $\gamma$-strongly convex with $\gamma=\lambda$, $L$-Lipschitz with $L=R_x^2(1+R_\theta)+\lambda R_\theta$, and $\beta$-smooth with $\beta=R_x^2+\lambda$.

The additional condition $\gamma/\beta \gg \frac{\log n}{n^{1-a}}$ can still hold if $R_x$, $R_\theta$, and $\lambda^{-1}$ vary as small polynomials of $n$.
\end{example}

Next, in order to provide a bound on the second order stability, we need to control the ``difference in difference'' of the loss functions, which inevitably involves the second order Taylor expansion and would require regularity of the Hessian of $\psi$.
\begin{assumption}[Lipschitz Hessian]\label{asm:sgd2}
For all $x$, $\ddot\psi(\cdot;x)$ is $L_2$-Lipschitz.
\end{assumption}
To the best of our knowledge, \Cref{asm:sgd2} has not appeared in the literature of stability of M-estimators. It is required to establish second order stability and we believe this condition is also needed for the second order stability proof in \cite{austern2020asymptotics}. 

Now we are ready to present our second order stability result for SGD.
\begin{proposition}[Second order stability of SGD]\label{prop:sgd2}
    Under  \Cref{asm:sgd2} and the same conditions as \Cref{prop:sgd1},
    \[
        \left\Vert \nabla_j \nabla_i \hat\theta_n \right\Vert \leq c(L, L_2, a, \gamma, \beta)n^{-2a} \log n\,.
    \]
\end{proposition}
The constant $c(L,L_a,a,\gamma,\beta)$ is explicitly tracked in the proof. The rate of $n^{-2a}$ is expected, while the additional $\log n$ factor arises in the proof when using a union bound on exponential tails.

For ridge regression, the additional requirement of \Cref{asm:sgd2} is trivial, as $\ddot\psi(\cdot;x)$ is constant.  Here we consider a less trivial example.
\begin{example}
Consider logistic ridge regression: $x=(y,z)\in \mathbb B(R_x)$, and $\theta\in \mathbb B(R_\theta)$.
$$
\psi(\theta;z) = -yz^T\theta+ \log(1+e^{z^T \theta}) + \lambda \|\theta\|^2\,.
$$

Then $\psi$ is $\gamma$-strongly convex with $\gamma=2\lambda$. $\psi$ is $L$-Lipschitz with $L=R_x^2 + \frac{1}{R_\theta} \log(1 + e^{R_x R_\theta}) + \lambda R_\theta$, and $\beta$-smooth with $\beta= \frac{R_x }{R_\theta  } + \lambda$. $\ddot \psi(\theta)$ is $L_2$-Lipschitz with $L_2 = \frac{R_x^2 }{4 R_\theta }$.

The required $o(n^{-3/2})$ rate can be achieved when $a>3/4$ and $R_x$, $R_\theta$, $\lambda^{-1}$ grow slowly enough as $n$ increases.
\end{example}

\paragraph{From parameter stability to loss stability.}  The first and second order stability for SGD established in \Cref{prop:sgd1,prop:sgd2} are for the estimated parameter $\theta$.  The corresponding stability conditions 
for the loss function $\ell(x;\theta)$ can be derived under suitable smoothness conditions on $\ell$.  In particular, if we assume that for each $x$, both $\ell(x;\theta)$ and $\dot\ell(x;\theta)$ are Lipschitz, then, using 
a standard first order Taylor expansion of $\ell(x;\theta)$ with respect to $\theta$, we have
\begin{align*}
    \left|\nabla_i \ell(x;\hat\theta_n)\right|\le C_1 \left\|\nabla_i\hat\theta_n\right\|\,,~~~~
    \left|\nabla_j\nabla_i \ell(x;\hat\theta_n)\right|\le  C_2 \left(\left\|\nabla_j\nabla_i\hat\theta_n\right\|+\sup_i\left\|\nabla_i\hat\theta_n\right\|^2\right)\,,
\end{align*}
where $C_1$ is a constant depends on the Lipschitz constant of $\ell(x;\cdot)$, and $C_2$ is a constant depends on the Lipschitz constants of $\ell(x;\cdot)$ and $\dot\ell(x;\cdot)$.  As a result, when $a$ is large enough and the Lipschitz constants in $\psi(x;\cdot)$, $\ell(x;\cdot)$ and their derivatives do not grow too fast, the first and second order loss stability conditions can be satisfied.

\subsection{Stability of difference loss: a nonparametric regression example}
Establishing the stability conditions for difference losses as required in \Cref{pro:cvc_coverage} is more delicate, mainly because \Cref{ass:lstability}(c) will be violated. For the Gaussian comparison argument to work, one needs to re-standardize the loss functions by considering
$$
\frac{\nabla_i(\ell_r-\ell_s)}{\sqrt{{\rm Var}(\ell_r-\ell_s)}}~~~~\text{and}~~~~
\frac{\nabla_j\nabla_i(\ell_r-\ell_s)}{\sqrt{{\rm Var}(\ell_r-\ell_s)}}
$$ 
in \Cref{ass:lstability}(a) and \ref{ass:lstability}(b), respectively.  In general, the variance of $\ell_r-\ell_s$ depends on the bias-variance decomposition of the two models, which may cancel with each other at different rates.  A general treatment of such a variance of difference loss functions seems to be a challenging task, if at all possible, and we will not pursue this in the current paper.  Instead, we use a simple non-parametric example to explicitly quantify the variance and stability of difference loss functions.
% Additionally, in order to provide an example for which the loss differences satisfy \Cref{ass:lstability}, we will consider under-parameterized linear models with decreasing coefficient magnitudes. This specific regime is chosen so to ensure that second moment of the loss difference vanishes at a rate slower than $n^{-1/2}$, as acknowledged in \Cref{rem:lossdiff}.

In particular, consider the following linear regression model with an infinite dimensional feature vector.
Let $(X_i:1\le i\le n)$ be iid vectors with $X_i=(Y_i,Z_i)\in\mathbb R^1\times\mathbb R^{\infty}$, such that
\begin{equation}\label{eq:regression_example}
Y_i=\sum_{j=1}^\infty \beta_j Z_{ij} +\epsilon_i
\end{equation}
where
\begin{enumerate}
    \item $\epsilon_i$ is independent of $Z_i$, has mean zero, finite variance $\sigma_\epsilon^2$, and is $1$-SW;
    \item  $Z_{ij}$ is mean zero, has unit variance, is $1$-SW, and satisfies $\mathbb E(Z_{ij}|Z_{i1},...,Z_{i,j-1})=0$ for all $j$;
    \item $\beta_j\asymp j^{-\frac{1+a}{2}}$ for some $a>0$ for all $j$\,.
\end{enumerate}
The zero mean and unit variance will simplify the estimator.  The sub-Weibull conditions will be carried over to the difference of the loss functions.  The rate of $\beta_j$ allows us to track the bias and variance precisely.

For each $j$, we can naturally estimate $\beta_j$ by
$$
\hat\beta_j=\frac{1}{n}\sum_{i=1}^n Y_i Z_{ij}\,.
$$
For each candidate model $r$, let $J_r=J_{r,n}$ be a positive integer such that $J_{n,r}\rightarrow\infty$ at a speed that is specified by the model $r$.  For example $J_{r,n}=n^{\alpha_r}$, where $\{\alpha_r: r\in\mathcal M\}$ is a collection of positive numbers corresponding to different increase speed of $J_{r,n}$.  The estimates we consider are truncated sequence estimates defined by integers $J_r$,
\begin{align*}
    & \hat f_r(Z_0) = \sum_{j=1}^{J_r} \hat \beta_j Z_{0j}\,,~~\forall~~r\in\mathcal M\,.
\end{align*}

\begin{proposition}[Loss difference stability in the nonparametric regression example]\label{prop:lossstab}
Under the regression model \eqref{eq:regression_example} and the conditions for $\epsilon$, $Z$, $\beta$,
if $J_r\le c J_s$ for some constant $c\in(0,1)$, then
\begin{itemize}
    \item[(a)] the first order stability required in \Cref{pro:cvc_coverage} holds when $J_s J_r^{a/2} = o(n)$;
    \item[(b)] the second order stability required in \Cref{pro:cvc_coverage} holds when $J_s J_r^{a/2} = o(\sqrt{n})$.
\end{itemize}
\end{proposition}

\begin{remark}
    Under the specified decay speed $\beta_j\asymp j^{-1-a}$, the optimal truncation $J^*$ that minimizes the risk under the squared prediction loss is $J^*\asymp n^{\frac{1}{1+a}}$.  If $J_s\gg J_r\asymp J^*$, the the condition $J_sJ_r^{a/2}=o(\sqrt{n})$ cannot be satisfied.  This implies the stability condition is harder to satisfy if the estimate from model $s$ has large variance.  On the other hand, if $J^*\asymp J_s$, then the condition  $J_sJ_r^{a/2}=o(\sqrt{n})$ can be satisfied when $a>1$ and $J_r\lesssim n^{(a-1)/(a+a^2)}$.  Again, this agrees with the intuition that cross-validation is able to screen out inferior models that have large bias \cite{yang2007consistency}. 
\end{remark}

\section{On deterministic centering}\label{sec:thm3}
%!TEX root = ./hd_cv_clt.tex
So far, we have focused on the high dimensional Gaussian comparison of cross-validated risk with random centering, where the mean vector $\tilde{\mathbf R}$ is data-dependent.   This leads to the following question: can we establish Gaussian comparison results with fixed centering?  It is natural to expect the fixed centering to be $\mathbf R^*=\mathbb E(\hat{\mathbf R}_{\rm cv})$.  Also, the  corresponding scaling should be based on the total variance $\mathbb E \left[\boldsymbol{\ell}(X_0;\mathbf X_{-1})\boldsymbol{\ell}(X_0;\mathbf X_{-1})^T\right]-\mathbf R^*{{\mathbf R}^*}^T$, which in addition to the variance term $\Sigma$ considered in the random centering case above, must also take into account the variability caused by the randomness of $\mathbf R(\mathbf X_{-1})$.  Such a fixed centering central limit theorem for cross-validated risks has been studied in \cite{austern2020asymptotics} in the low-dimensional case.  Our development here extends their result to the high-dimensional case with a more streamlined proof.

\subsection{Risk stability}
In order to study the randomness in the risk function $R_r(\mathbf X_{-1})$,
we need the following stability conditions on the risk functions.
\begin{assumption}[Risk stability]\label{ass:risk_stab}
There exists a constant $\rstability\in\mathbb R^+$ such that % and $\rstabilitya\in(0,1)$
$n \nabla_i R_r(\mathbf X_{-1})$ is $\rstability$-SW for all $i$.
% \begin{enumerate}
% 	\item $n \nabla_i R_r(\mathbf X_{-1})$ is $\rstability$-SW for all $i$, and
% 	\item $n^{3/2}\nabla_j\nabla_i R_r(\mathbf X_{-1})$ is $\rstabilitya$-SW for $i\neq j$.
% \end{enumerate}
\end{assumption}
A remarkable difference between the risk stability and loss stability is in the scaling factors.
For the first order differencing operators $\nabla_i$, the scaling factor of $n$ instead of $\sqrt{n}$ makes the risk stability apparently harder to control than the loss stability.  This is also considered and briefly discussed in \cite{austern2020asymptotics} in the low-dimensional case.  Here, we give a detailed explanation of this key condition.
 At first, it seems unreasonable to assume that $\rstability$ is very small, as $\nabla_i R_r(\mathbf X_{-1})$ generally should not be smaller than $1/n$. However, a closer inspection suggests that the risk function (taking expectation of the loss function $\ell_r(X_0,\mathbf X_{-1})$ over the evaluating point $X_0$) is usually much more stable than the loss function itself, as taking a conditional expectation usually increases stability. In fact, such an increase of stability can be quite substantial.
For example, assume that the loss function takes a parametric form: $\ell(x_0;x_1,...,x_{\tilde n})=\ell(x_0;\hat\theta)$ where $\hat\theta=\hat\theta(x_1,...,x_{\tilde n})$ is a fitted parameter from the input data $(x_1,...,x_{\tilde n})$. Then $$\nabla_i R(\hat\theta)\approx \left(\frac{d R}{d\theta}\bigg|_{\hat\theta}\right) \nabla_i\hat\theta\,,$$
which should be much smaller than $\nabla_i \hat\theta$ if $R(\theta)\coloneqq \mathbb E\ell(X_0;\theta)$ is flat at $\hat\theta$.  This is usually the case when $\hat\theta$ is in  a small neighborhood of the optimal parameter value with minimum risk $\theta^*$.

Furthermore, we remark that our theoretical development does not require $\rstability$ to vanish asymptotically. Instead, we only need $\rstability$ to be dominated by other vanishing terms such as $\lstability$ and $1/\sqrt{n}$.
% On the other hand, our theory does require $\rstabilitya$ to vanish as $(n,p)\rightarrow\infty$ in order for Gaussian comparison to hold.  This is less restrictive since taking one difference typically reduces the magnitude of the random variable by a factor between $n^{1/2}$ and $n$.
\subsection{Gaussian comparison with deterministic centering}
Finding the covariance matrix for deterministic centering starts by identifying the contribution of randomness from each single sample point.  We start by writing
 $$n\hat R_{{\rm cv},r}=\sum_{i=1}^n\ell_r(X_i;\mathbf X_{-v_i})\,.$$  
 The part in the above sum that involves $X_i$ is
\begin{align}
\ell_r(X_i;\mathbf X_{-v_i})+\sum_{j\notin I_{v_i}}\ell_r(X_j;\mathbf X_{-v_j})\,.\label{eq:Xi_in_Rcv}
\end{align}
It is clear that $X_i$ plays two different roles in $\hat R_{{\rm cv},r}$: (i) as the evaluation point in $\ell_r(X_i;\mathbf X_{-v_i})$, (ii)
as one of the $\tilde n$ fitting sample points in each of $\ell_r(X_j;\mathbf X_{-v_j})$ for $j\neq I_{v_i}$.
The randomness contributed by $X_i$ as an evaluating point should be captured by the variability of the average loss function
$$
\bar\ell_r(X_1) = \mathbb E\left[\ell_r(X_1;\mathbf X_{-1})|X_1\right]\,,
$$
and the randomness contributed by $X_1$ as a fitting sample point should be captured in the function
$$
\bar R_r(X_1) = \mathbb E\left[R(\mathbf X_{-V})|X_1\right]\,.
$$

Therefore, let $\boldsymbol{\Phi}=(\phi_{rs}:1\le r,s\le p)$ be the covariance matrix given by
\begin{align}
\phi_{rs}\coloneqq  & {\rm Cov}\left[\bar\ell_r(X_1)+\tilde n \bar R_r(X_1)\,,~\bar\ell_s(X_1)+\tilde n \bar R_s(X_1)\right]\,~~1\le r,s\le p\,.\label{eq:phi_rs}
\end{align}
% with the following variance term notation for convenience
% \begin{align*}
% \phi_{r}^2\coloneqq  & \phi_{rr}\,.
% \end{align*}
We assume that the marginal variance terms are bounded and bounded away from $0$, leading to the following assumption which is analogous to \Cref{ass:condition_l}(b).
\begin{assumption}\label{ass:phi_bound}
	There exist positive constants $\underline\phi$ and $\bar \phi$ such that $\underline\phi\le \phi_{rr} \le\bar\phi$ for each $r\in[p]$.
\end{assumption}
Using this assumption on the marginal variance terms, the symmetry and moment condition on $\ell_r$, and all previous stability assumptions on $\ell_r$ and $R_r$, we state the following result for the Gaussian comparison with fixed centering.
% In the following theorem we will proceed with the requirement of $\lstability\le 1$ for notational simplicity and without loss of generality, because  the approximation error bound is meaningless if $\lstability>1$.
\begin{theorem}[Deterministic Gaussian Comparison]\label{thm:thm3}
	Assume  \Cref{ass:condition_l,ass:lstability,ass:risk_stab,ass:phi_bound} hold, then we have
	\begin{align*}
		&\sup_{z\in\mathbb R}\left|\mathbb P\left(\sqrt{n}\max (\hat{\mathbf R}_{\rm cv}-\mathbf R^*)\le z\right)-\mathbb P\left(\max \mathbf Y\le z\right)\right|\\
		&\le \tilde O \left([\lstability(1+\rstability)]^{1/3} + n^{-1/8}(1+\rstability)^{3/4}\right)
	\end{align*}
for $\mathbf Y\sim N(0,\boldsymbol{\Phi})$.
\end{theorem}
\begin{remark}Similar to \Cref{rem:constant_sigma}, the constants $\bar\phi$ and $\underline{\phi}$ appear in the error bound as a multiplicative factor in the $\tilde O(\cdot)$ notation, which is no larger than $\bar\phi \underline{\phi}^{-2}$.
\end{remark}

\subsection{Deterministic variance estimation}
Finally, we address the problem of estimating $\phi_{rs}$, which has also been considered in \cite{austern2020asymptotics}.  We believe the estimate stated in their text is off by a factor of 2, and also only covers the case of two-fold cross-validation.  Our result below corrects the scaling and covers the general $V$-fold case in a multivariate setting.

As suggested in \eqref{eq:phi_rs} and \Cref{thm:thm3}, the covariance $\phi_{rs}$ is essentially the sum of the marginal variability of each $X_i$.  Indeed, we have the following result
\begin{theorem}[Marginal variance approximation]\label{thm:marginal_var}
Under \Cref{ass:condition_l,ass:lstability,ass:risk_stab,ass:phi_bound}, we conclude that
	$$\frac{n^2}{2}\mathbb E\left[\nabla_i \hat R_{{\rm cv},r}\nabla_i\hat R_{{\rm cv},s}|\mathbf X_{-v_i}\right]-
\phi_{rs}$$ 
is $\lstability(1+\rstability)$-SW.
\end{theorem}

\Cref{thm:marginal_var} implies that we can simply estimate
$\mathbb E[\nabla_1 \hat R_{{\rm cv},r}\nabla_1 \hat R_{{\rm cv},s}|\mathbf X_{v_1}]$ to approximate $\phi_{rs}$.
This leads to the following procedure, which requires a hold-out set of iid sample points $X_1',...,X_m'$ from the same distribution, that are not involved in any cross-validation folds.  In practice, one can choose a small but diverging value of $m=n^a$ with $a\in (0,1)$, then use $n-m$ sample points for the $V$-fold cross-validation and $m$ hold-out sample points for variance estimation.

For $i\in [n]$ and $j\in[m]$, define
$\hat{\mathbf R}_{{\rm cv}}^{i,j}$
to be the cross-validation risk vector obtained by replacing $X_i$ with $X_j'$.
Then \Cref{thm:marginal_var} implies the following.
\begin{corollary}\label{cor:det_cov_est_consist}
Define
\begin{align}
	\hat\phi_{rs} = \frac{n^2}{m}\sum_{j=1}^{m/2}\left( \hat R_{{\rm cv},r}^{1,2j-1} - \hat R_{{\rm cv},r}^{1,2j}\right)\left(\hat R_{{\rm cv},s}^{1,2j-1}-\hat R_{{\rm cv},s}^{2j}\right)\,.\label{eq:thm3_var_estimator}
\end{align}
	Then with probability at least $1-O((n+p)^{-1})$ we have
	$$
\sup_{r,s}|\hat\phi_{rs}-\phi_{rs}|\le \tilde O\left(\lstability(1+\rstability)+m^{-1/2}\right)\,.
	$$
\end{corollary}
The estimator in \eqref{eq:thm3_var_estimator} estimates $\mathbb E[\nabla_1 \hat R_{{\rm cv},r}\nabla_1 \hat R_{{\rm cv},s}|\mathbf X_{-v_1}]$ by taking empirical average over $m/2$ conditional iid samples given the fitting data $\mathbf X_{-v_1}$, which is supported by \Cref{thm:asymp_var_nabla_i} and \Cref{thm:concentration_nabla_i_f}.  In practice, we can possibly also use 
$$
\hat\phi_{rs} = \frac{n^2}{2m}\sum_{j=1}^{m}\left( \hat R_{{\rm cv},r} - \hat R_{{\rm cv},r}^{i_j,j}\right)\left(\hat R_{{\rm cv},s}-\hat R_{{\rm cv},s}^{i_j,j}\right)\,.
$$
which perturbs different entries instead of just the first one.

\Cref{cor:det_cov_est_consist} provides an entry-wise error bound of the covariance estimation, which is good enough for Gaussian approximation of the supremum, as demonstrated in \Cref{cor:simul_ci}.  The same kind of inference procedures considered in \Cref{sec:simu_band} can be carried over to the deterministic centering case, which is omitted here as there is little additional insight.

\section{Discussion}
%!TEX root = ./hd_cv_clt.tex
Since its first appearance, high-dimensional Gaussian comparisons have found wide applications in statistical inference problems, and have been extended and improved by many authors.  In addition to the extensions to dependent data mentioned above, sharper results on the Gaussian comparison of independent sums have been obtained in recent literature. For example, see \cite{deng2020beyond,kuchibhotla2021high,lopes2020central,kuchibhotla2020high}.  In our work, the goal is to develop an asymptotic Gaussian comparison to serve the purpose of statistical inference. Thus we did not attempt to obtain the optimal Berry-Esseen type of  convergence rates.  Our proof uses the Slepian interpolation as in the original work \cite{chernozhukov2013gaussian}, and it seems possible to obtain better rates of convergence if more refined techniques are used.

Our main motivations are to understand the joint randomness of many cross-validated risks, and to provide theoretical foundations for uncertainty quantification of cross-validation based model selection.  The theory included in this work is particularly relevant to the ``cross-validation with confidence'' method \citep{lei2020cross}, where one uses the asymptotic Gaussian comparison to construct a confidence set that contains the best model with a prescribed confidence level.  This method is connected to the model selection confidence set literature \citep{hansen2011model}, which has largely relied on sequential hypothesis testing based approaches \citep{gunes2012confidence,ferrari2015confidence,jiang2008fence}.  We expect the theory outlined in this paper to be useful in developing a new model confidence set estimator using cross-validation with both provable validity guarantees and good practical performance.

\appendix
\section{More notation, definition, and basic properties}
%!TEX root = ./hd_cv_clt.tex

\subsection{Definition and properties of sub-Weibull concentration}
\begin{definition}[sub-Weibull]\label{def:sw_new}
Let $K,\alpha$ be positive numbers. We say a random variable $X$ is $(K,\alpha)$-sub-Weibull (or $(K,\alpha)$-SW)  if any of the following holds:
\begin{enumerate}
	\item There exists constant $a$ such that $\mathbb P\left(\frac{|X|}{K} \ge t\right)\le a e^{-t^\alpha}$, for all $t>0$.
	\item There exists constant $c$ such that $\|X\|_q\le cK q^{1/\alpha}$ for all $q\ge 1$. 
\end{enumerate}
\end{definition}
The equivalence of these two definitions can be found in, for example, Theorem 2.1 of \cite{vladimirova2020sub}. The constants $a,c$ in the definition above are not important, and are used here so that the two definitions have the same $(K,\alpha)$ pair.

\begin{proposition}[Basic properties of sub-Weibull random variables]\label{pro:basic_SW}
If $X_i$ is $(K_i,\alpha_i)$-SW, for $i=1,2$, then
\begin{enumerate}
	\item  $X_1 X_2$ is $(K_1K_2,\frac{\alpha_1\alpha_2}{\alpha_1+\alpha_2})$-SW.
	\item $X_1+X_2$ is $(K_1\vee K_2,\alpha_1\wedge \alpha_2)$-SW.
\end{enumerate}
\end{proposition}

The following theorem controls the tail integral of sub-Weibull random variables.
\begin{lemma}\label{lem:subWeibull-tail-int}
	If $Y$ is $(K,\alpha)$-sub-Weibull, then there exists constant $c>0$ independent of $K$ such that 
	$\mathbb E[|Y|\mathds{1}(|Y|\ge w K)]\le c K w\exp(-w^\alpha)$ for any $w\ge 1$.
\end{lemma}
\begin{proof}[Proof of \Cref{lem:subWeibull-tail-int}]
	Without loss of generality, assume $Y\ge 0$ and $K=1$. Let $f(y)$ be the density function of $Y$.
	\begin{align*}
		\mathbb E[Y\mathds{1}(Y\ge w)]=&\int_{y=w}^\infty y f(y)dy\\
=& \int_{y=w}^\infty \int_{u=0}^y du f(y) dy\\
=&\int_{u=0}^w \int_{y=w}^\infty f(y)dy du+\int_{u=w}^\infty \int_{y=u}^\infty f(y)dy du\\
=& w \mathbb P(Y\ge w) + \int_{u=w}^\infty \mathbb P(Y\ge u) du\\
\le & w a \exp(-w^\alpha) +\int_{u=w}^\infty a \exp(-u^\alpha)du\\
= &w a \exp(-w^\alpha) +\frac{a}{\alpha}\int_{v=w^\alpha}^\infty \exp(-v)v^{\frac{1}{\alpha}-1}dv\,.
	\end{align*}
	When $\alpha\in(0,1]$, since $w\ge 1$ we have, by \cite[][Proposition 4.4.3]{gabcke1979neue}
	\begin{align*}
		\int_{v=w^\alpha}^\infty \exp(-v)v^{\frac{1}{\alpha}-1}dv\le \frac{1}{\alpha}e^{-w^\alpha} w^{1-1/\alpha}\,.
	\end{align*}
	Thus, $\mathbb E[Y\mathds{1}(Y\ge w)]\le a(1+\alpha^{-2}) w \exp(-w^\alpha)$.

	When $\alpha>1$, $v^{1/\alpha-1}\le 1$ on $[w,\infty)$ since $w\ge 1$, so we have
	\begin{align*}
		\int_{v=w^\alpha}^\infty \exp(-v)v^{\frac{1}{\alpha}-1}dv\le \exp(-w^\alpha)\,.
	\end{align*}
	So
	$\mathbb E[Y\mathds{1}(Y\ge w)]\le 2aw \exp(-w^\alpha)$\,.
\end{proof}

The following lemma is a sub-Weibull version of martingale concentration inequality, showing that the scaling of a martingale with stationary sub-Weibull increments scales at the speed of $\sqrt{n}$, where $n$ is the time horizon.
\begin{lemma}\label{lem:MG-subW}
Let $M=\sum_{i=1}^n M_i$ where the sequence $(M_i)_{i=1}^n$ satisfies
\begin{enumerate}
	\item martingale property: $\mathbb E(M_i|M_j:1\le j< i)=0$ for all $2\le i\le n$, and $\mathbb E M_1=0$.
	\item sub-Weibull tail: $\|M_i\|_q\le cK_i q^{1/\alpha_i}$ for some $c,\alpha_i>0$ and all $q\ge 1$. 
\end{enumerate}                 
Then we have, for $\alpha'=\frac{2\alpha^*}{2+\alpha^*}$ with $\alpha^* = \min_{j\leq n}\alpha_j$ and a positive constant $c'$,
$$
\|M\|_q\le c' \left(\sum_{i=1}^n K_i^2\right)^{1/2} q^{1/\alpha'}\,,~~\forall q\ge 1\,.
$$  
If $K_i = K$ for all $i\in[n]$, then
$$
\|M\|_q\le c' \sqrt{n}K q^{1/\alpha'}\,,~~\forall q\ge 1\,.
$$  
\end{lemma}
\begin{proof}[Proof of \Cref{lem:MG-subW}]
	By Theorem 2.1 of \cite{rio2009moment}, we have for any $q\ge 2$
	\begin{align*}
		\left\|M\right\|_q\le \left[(q-1)\sum_{i=1}^n \|M_i\|_q^2\right]^{1/2}\le \left[C(q-1)q^{2/\alpha^*}\sum_{i=1}^n K_i^2\right]^{1/2}\le C^{1/2} q^{\frac{2+\alpha^*}{2\alpha^*}}\left( \sum_{i=1}^n K_i^2 \right)^{1/2}\,.
	\end{align*}
	where $C$ is a constant depending only on $c$, and the second inequality follows from the assumption $\|M_i\|_q\le c K_i q^{1/\alpha_i}$.
\end{proof}

\section{Proof for random centering}
%!TEX root = ./hd_cv_clt.tex
\subsection{Notation}
We first collect some notation for the proof.
For the ease of presentation, we use $\mathbf W=(W_r:1\le r\le p)$ to denote the centered and scaled random vector for which we would like to establish a normal approximation.
Thus, in the proof that follows, the symbol $W_r$ may refer to different objects than in other proofs in this work. 
In particular, for the random centering/scaling case (\Cref{thm:1}), $W_r
=\sqrt{n}(\hat R_{{\rm cv},r}-\tilde R_r)$, while in the deterministic centering/scaling case (\Cref{thm:thm3}), $W_r=\sqrt{n}(\hat R_{{\rm cv},r}-R_r^*).$

Recall the following notation:
\begin{itemize}
	\item $\mathbf X_{-v}$: the $\tilde n$ ($=n(1-1/V)$) subvector of $\mathbf X$ excluding those in index $I_{v}$.
	\item $\mathbf X^{-i}$: the $(n-1)$ subvector of $\mathbf X$ excluding the $i$th entry.
	\item $\mathbf X^i$: the iid vector of $\mathbf X$ with $i$th entry being $X_i'$, an iid copy of $X_i$. 
\end{itemize}

For random objects $(U,V)$ and function $f$ acting on $(U,V)$, we will also use the notation $\mathbb E_U f(U,V) = \mathbb E[f(U,V)|V]$. For example $\mathbb E_{X_i}f(\mathbf X)=\mathbb E[f(\mathbf X)|\mathbf X^{-i}]$.

\subsection{Proof of \Cref{thm:1}}
\begin{proof}[Proof of \Cref{thm:1}]
	Recall the notation: 
	$$W_r=\sqrt{n}(\hat R_{{\rm cv},r}-\tilde R_r)=\sum_{i=1}^n \frac{1}{\sqrt{n}}K_{r,i}\,.$$
	Consider the leave-one-out version of $W_r$:
	$$
      W_{r}^i=\sum_{j\neq i} K_j(\mathbf X^i)/\sqrt{n} = W_r-\frac{1}{\sqrt{n}} K_{r,i}-\frac{1}{\sqrt{n}}D_{r,i}\,.
	$$
	The plan is to use Slepian's interpolation which smoothly bridges between $\mathbf W$ and the corresponding Gaussian vector $\mathbf Y$.
In order to do so, we consider an intermediate object 
$$\hat{\mathbf Y}=\frac{1}{\sqrt{n}}\sum_{i=1}^n \hat{\mathbf Y}_i\coloneqq\frac{1}{\sqrt{n}}\sum_{i=1}^n \Sigma_{v_i}^{1/2}\varepsilon_i$$ 
with $\varepsilon_i\stackrel{iid}{\sim}N(0,I_p)$, $\Sigma_{v}=[\sigma_{rs,v}]_{1\le r,s\le p}$ being the conditional covariance matrix of $\mathbf K(X_0;\mathbf X_{-v})$ given the fitting data $\mathbf X_{-v}$, and $v_i$ the fold identifier of the sample point indexed by $i$.

	Define the interpolating vector, for $t\in (0,1)$
	$$Z_r(t)=\sqrt{t}W_r + \sqrt{1-t} \sum_i \hat Y_{r,i}/\sqrt{n}\,,$$
	and the corresponding leave-one-out version
    \begin{align*}
    	Z_r^i(t) = \sqrt{t}W_{r}^i+\sqrt{1-t}\sum_{j\neq i} \hat Y_{r,i}/\sqrt{n}
    \end{align*}
    which satsifies
    $$
     Z_r(t)-Z_r^i(t) = Z_{r,i}(t)+D_i(t)
    $$
    with
    \begin{align*}
    	Z_{r,i}(t) = & \sqrt{t} K_i/\sqrt{n} + \sqrt{1-t}\hat Y_{r,i}/\sqrt{n}\,,\\
    	D_{r,i}(t) = & \sqrt{t}D_{r,i}/\sqrt{n}\,.
    \end{align*}

   Let $h:\mathbb R^p\mapsto \mathbb R$ be such that for all $z\in\mathbb R$. Define
   \begin{align}
    	\sum_{r,s=1}^p |\partial_r\partial_s h(z)|= & M_2(h)\,,\label{eq:M2}\\
        \sum_{r,s,u=1}^p |\partial_r\partial_s\partial_u h(z)|= &M_3(h)\,.\label{eq:M3}
    \end{align}

Because $\mathbb E [h(\mathbf W)-h(\hat{\mathbf Y})] =\mathbb E \int_0^1 \frac{dh(\mathbf Z(t))}{dt}dt$, the main step in the proof is to control $\mathbb E \frac{d h(\mathbf Z(t))}{dt}$.

   By Taylor expansion:
   \begin{align}
	\mathbb E\frac{dh(\mathbf Z(t))}{dt}& = \sum_{r=1}^p \sum_{i=1}^n \mathbb E[\partial_r h(\mathbf Z^i(t)) Z^\prime_{r,i}(t)]\nonumber \\
	&~~ + \sum_{s=1}^p \sum_{r=1}^p \sum_{i=1}^n \mathbb E[\partial_s \partial_r h(\mathbf Z^i(t)) \left( Z_{s,i}(t) + D_{s,i}(t) \right) Z_{r,i}^\prime(t)] \nonumber\\
	& ~~ + \sum_{u=1}^p \sum_{s=1}^p \sum_{r=1}^p \sum_{i=1}^n \mathbb E\bigg\{ [Z_{s,i}(t) + D_{s,i}(t)] [Z_{u,i}(t) + D_{u,i}(t)]\nonumber\\ 
	&~~~~\times\left[ \int_0^1 (1-v) \partial_u \partial_s \partial_r h(\mathbf Z^i(t) + v\mathbf Z_i(t)) dv \right] Z^\prime_{r,i}(t) \bigg\}.\label{eq:thm1_main_expansion}
\end{align}

The first term in \eqref{eq:thm1_main_expansion}
\begin{align*}
	\mathbb E[\partial_r h(\mathbf Z^i(t)) Z^\prime_{r,i}(t)] = \mathbb E\mathbb E_{X_i,\varepsilon_i}[\partial_r h(\mathbf Z^i(t)) Z^\prime_{r,i}(t)]=\mathbb E\left\{\partial_r h(\mathbf Z^i(t))\mathbb E_{X_i,\varepsilon_i}[Z^\prime_{r,i}(t)]\right\}=0\,.
\end{align*}

The second term in \eqref{eq:thm1_main_expansion} consists of two parts. First,
\begin{align*}
	&\mathbb E[\partial_s \partial_r h(\mathbf Z^i(t)) Z_{s,i}(t) Z_{r,i}^\prime(t)]\\
	=& (2n)^{-1}\mathbb E\partial_s \partial_r h(\mathbf Z^i(t)) \left[(\sqrt{t}K_{s,i}+\sqrt{1-t}\hat Y_{s,i})\left(\frac{K_{r,i}}{\sqrt{t}}-\frac{\hat Y_{r,i}}{\sqrt{1-t}}\right)\right]\\
	=& (2n)^{-1}\mathbb E\left\{\partial_s \partial_r h(\mathbf Z^i(t))\mathbb E_{X_i,\varepsilon_i} \left[(\sqrt{t}K_{s,i}+\sqrt{1-t}\hat Y_{s,i})\left(\frac{K_{r,i}}{\sqrt{t}}-\frac{\hat Y_{r,i}}{\sqrt{1-t}}\right)\right]\right\}\\
	=&0\,,
\end{align*}
by construction of $\hat{\mathbf Y}_i$.
Now the second term in \eqref{eq:thm1_main_expansion} reduces to
\begin{align*}
	\sum_{s,r}\sum_i\mathbb E[\partial_s \partial_r h(\mathbf Z^i(t)) D_{s,i}(t) Z_{r,i}^\prime(t)]\,.
\end{align*}
By \Cref{lem:D}, $D_{s,i}$ is $\lstability$-SW.
In $Z_{r,i}'(t)=(K_{r,i}/\sqrt{t}-\hat Y_{r,i}/\sqrt{1-t})/(2\sqrt{n})$, $K_{r,i}$ is $1$-SW by \Cref{ass:lstability}, and $\hat Y_{r,i}\stackrel{d}{=}\sigma_{rr,v_i}\varepsilon_{r,i}$ is also $1$-sub-Weibull as $\sigma_{r,v_i}$ is $1$-sub-Weibull according to the proof of \Cref{lem:cov_concentration}.
 Therefore, $D_{s,i}(t) Z_{r,i}^\prime(t)$ is $n^{-1}\lstability \eta_t$-SW, where
\begin{equation}\label{eq:eta_t}
\eta_t=t^{-1/2}\vee (1-t)^{-1/2}\,.\end{equation}
Now for any $\tau >0$,  by \Cref{lem:subWeibull-tail-int}
\begin{align}
	&\sum_{s,r}\sum_i\mathbb E[\partial_s \partial_r h(\mathbf Z^i(t)) D_{s,i}(t) Z_{r,i}^\prime(t)]\label{eq:SW_application_thm1q2}\\
	= & \sum_{s,r}\sum_i\mathbb E[\partial_s \partial_r h(\mathbf Z^i(t)) D_{s,i}(t) Z_{r,i}^\prime(t)\mathds{1}(|D_{s,i}(t) Z_{r,i}^\prime(t)|\le n^{-1}\tau \lstability\eta_t)]\nonumber\\
	&~~+\sum_{s,r}\sum_i\mathbb E[\partial_s \partial_r h(\mathbf Z^i(t)) D_{s,i}(t) Z_{r,i}^\prime(t)\mathds{1}(|D_{s,i}(t) Z_{r,i}^\prime(t)|> n^{-1}\tau \lstability\eta_t)]\nonumber\\
	\le & n^{-1}\tau \lstability\eta_t\sum_{s,r}\sum_i\mathbb E|\partial_s \partial_r h(\mathbf Z^i(t))| \nonumber \\
	 &~~+M_2\sum_{s,r}\sum_i\mathbb E[|D_{s,i}(t) Z_{r,i}^\prime(t)|\mathds{1}(|D_{s,i}(t) Z_{r,i}^\prime(t)|> n^{-1}\tau  \lstability\eta_t)]\nonumber\\
	 \lesssim & \tau \lstability \eta_t M_2+ np^2 e^{-\tau^c}\,.\nonumber
\end{align}
By choosing $\tau=c_1\log^{c_2}(n+p)$ with appropriate choices of constants $c_1,c_2$ independent of $(n,p)$, \eqref{eq:SW_application_thm1q2} is bounded by $\tilde O(M_2 \lstability\eta_t)$.

The third term in \eqref{eq:thm1_main_expansion} is similarly controlled: let 
$$Q_{rsu,i}=\int_0^1 (1-v) \partial_u \partial_s \partial_r h(\mathbf Z^i(t) + v\mathbf Z_i(t)) dv$$
and
$$
T_{rsu,i} = [Z_{s,i}(t) + D_{s,i}(t)] [Z_{u,i}(t) + D_{u,i}(t)]Z^\prime_{r,i}(t)\,.
$$
By definition of $M_3$ we have
$$
\sum_{r,s,u} |Q_{rsu,i}|\le M_3
$$
and $T_{rsu,i}$ is $n^{-3/2}\eta_t$-SW. Thus the third term is controlled by
$$
\tilde O(n^{-1/2} M_3\eta_t)\,.
$$

Since $\eta_t$ is integrable on $(0,1)$, we have shown that
$$
|\mathbb E h(\mathbf W) - \mathbb E h(\hat{\mathbf Y})| \le \tilde O(\lstability M_2+n^{-1/2}M_3)\,.
$$

Combining \Cref{lem:h_function} and the anti-concentration result\footnote{The anti-concentration result there is for Gaussian processes.  However, our $\hat{\mathbf Y}$ is a Gaussian mixture because $\hat{\mathbf Y}$ is Gaussian only when conditioning on $\mathbf X$. We can condition on $\mathbf X$, provided that $\sup_{r,v}\sigma_{rr,v}\le \tilde O(1)$ with high probability.  
This can be established if $\sigma_{rr,v}$ is $1$-SW, which is implied by the proof of \Cref{lem:cov_concentration}.} 
\citep[][Lemma 2.1]{chernozhukov2013gaussian}, we have for any $0<\beta<n$
\begin{align}
\sup_{z}\left|\mathbb P\left(\max_r W_r \le z\right)-\mathbb P(\max_r \hat Y_r\le z)\right|
\le & \tilde O\left(\lstability\beta^2+n^{-1/2}\beta^3+\beta^{-1}\right)\nonumber\\
\le & \tilde O \left(\lstability^{1/3}\vee n^{-1/8}\right) \,.\label{eq:thm1_max_comp1}
\end{align}
where the last inequality follows by choosing $\beta=\min(\lstability^{-1/3},n^{1/8})$.

To get the final approximation, let
\begin{equation}\label{eq:def_Delta}
\Delta = \max_{r,s,v}|\sigma_{rs,v}-\sigma_{rs}|\,,
\end{equation}
and event
\begin{equation}\label{eq:def_E}\mathcal E=\{\Delta\le c_1\lstability\log^{c_2}(n+p)\}\,,\end{equation}
with appropriately chosen constants $c_1,c_2$ such that, according to \Cref{lem:cov_concentration},
$$
\mathbb P(\mathcal E) \ge 1-n^{-1}\,.
$$
Then
\begin{align}
	\mathbb P\left(\max_r \hat Y_r\le z\right) \le & \mathbb P\left(\max_r \hat Y_r \le z |\mathcal E\right)+\mathbb P\left(\mathcal E^c\right)\nonumber\\
\le & \mathbb P\left(\max_r Y_r\le z\right)+\tilde O(\lstability^{1/3}\log^{1/3}(\lstability^{-1}))+n^{-1}\,,\label{eq:thm1_max_comp2a}
\end{align}
where the last inequality uses Theorem 2 of \cite{chernozhukov2015comparison} between $\hat{\mathbf Y}$ and $\mathbf Y$.
On the other hand we have
\begin{align}
	\mathbb P\left(\max_r \hat Y_r\le z\right) \ge & \mathbb P\left(\max_r \hat Y_r \le z |\mathcal E\right)\mathbb P(\mathcal E)\nonumber\\
	\ge & \left[\mathbb P\left(\max_r Y_r\le Z\right)-\tilde O(\lstability^{1/3}\log^{1/3}(\lstability^{-1}))\right](1-n^{-1})\\
	\ge & \mathbb P\left(\max_r Y_r\le Z\right)-\tilde O(\lstability^{1/3}\log^{1/3}(\lstability^{-1}))\,.\label{eq:thn1_max_comp2b}
\end{align}
The claimed result follows by combining \eqref{eq:thm1_max_comp1}, \eqref{eq:thm1_max_comp2a}, and \eqref{eq:thn1_max_comp2b}.
\end{proof}

\subsection{Proof of \Cref{cor:simul_ci}}
\begin{proof}[Proof of \Cref{cor:simul_ci}]
% Let $\stackrel{\mathcal L}{\approx}$ denote approximately equal in distribution. Our plan is to establish the following chain of approximations.
% \begin{align}
% \left\|\frac{1}{\sqrt{n}}\hat\Lambda^{-1/2}(\hat{\mathbf R}_{\rm cv}-\bar{\mathbf R})\right\|_\infty\stackrel{\mathcal L}{\approx} &
% \left\|\frac{1}{\sqrt{n}}\Lambda^{-1/2}(\hat{\mathbf R}_{\rm cv}-\bar{\mathbf R})\right\|_\infty\nonumber\\
% \stackrel{\mathcal L}{\approx} & \left\|\frac{1}{\sqrt{n}}\Lambda^{-1/2}\sum_i \Theta_{v_i}^{1/2}\varepsilon_i\right\|_\infty\nonumber\\
% \stackrel{\mathcal L}{\approx} & \left\|\frac{1}{\sqrt{n}}\hat\Lambda^{-1/2}\hat\Theta^{1/2}\sum_i \varepsilon_i\right\|_\infty\label{eq:normalized_apprx_chain}
% \end{align}	
First, define event $\mathcal E_1$ on the space of $\mathcal X^n$ as the subset consisting all samples of size $n$ such that
\begin{align*}
\sup_v\|\hat\Sigma_v-\Sigma\|_\infty  \lesssim \tilde O\left(\frac{1}{\sqrt{n}}+\lstability\right)\,,
\end{align*}
where the constants $c_1,c_2$ in the $\tilde O$ notation is omitted. Then
combining \Cref{lem:cov_concentration} and standard sub-Weibull concentration of iid sums we have
$$
\mathbb P(\mathcal E_1)\ge 1-n^{-1}
$$
with appropriate choice of universal constants in $\tilde O(\cdot)$.  Here the $n^{-1/2}$ term comes from $\|\hat\Sigma_v-\Sigma_v\|_\infty$ and the $\lstability$ term comes from $\|\Sigma_v-\Sigma\|_\infty$.

Let 
\begin{align*}
	\delta_0=&\left\|\sqrt{n}\hat\Lambda^{-1/2}(\hat{\mathbf R}_{\rm cv}-\tilde{\mathbf R})\right\|_\infty\,,\\
\delta_1=&\left\|\sqrt{n}\Lambda^{-1/2}(\hat{\mathbf R}_{\rm cv}-\tilde{\mathbf R})\right\|_\infty\,,\\
\delta_2=&\|\mathbf Y\|_\infty\,,~~\mathbf Y\sim N(0,\Lambda^{-1/2}\Sigma\Lambda^{-1/2})\,,\\
\delta_3=&\|\tilde{\mathbf Y}\|_\infty\,,~~\tilde{\mathbf Y}\sim N(0,\hat\Lambda^{-1/2}\hat\Sigma\hat\Lambda^{-1/2})\,.
\end{align*}

On $\mathcal E_1$ we have
$$
|\delta_0-\delta_1|\le \tilde O(n^{-1/2}+\lstability)\,.
$$

Define $\mathcal E_2$ be the event that $\left\|\sqrt{n}(\hat{\mathbf R}_{\rm cv}-\tilde{\mathbf R})\right\|\le 2\sqrt{\log(n+p)}$. Then \Cref{thm:1} implies that
$$
\mathbb P(\mathcal E_2)\ge 1-n^{-1}-\tilde O(n^{-1/8}+\lstability^{1/3})\,.
$$
Then we have the following approximation.
\begin{align*}
\mathbb P( \delta_0\le t)\le & \mathbb P(\delta_0\le t\,,~\mathcal E_1\cap\mathcal E_2)+\mathbb P(\mathcal E_1^c)+\mathbb P(\mathcal E_2^c)\\
\le & \mathbb P(\delta_1\le t+|\delta_1-\delta_0|\,,~\mathcal E_1\cap\mathcal E_2)+\tilde O(n^{-1/8}+\lstability^{1/3})\\
\le & \mathbb P\left[\delta_1\le t+\tilde O(n^{-1/2}+\lstability)\right]+\tilde O(n^{-1/8}+\lstability^{1/3})\\
\le &\mathbb P\left[\delta_2\le t+\tilde O(n^{-1/2}+\lstability)\right]+\tilde O(n^{-1/8}+\lstability^{1/3})\\
\le & \mathbb P\left[\delta_3\le t+\tilde O(n^{-1/2}+\lstability)\big|\mathcal E\right] + \tilde O(n^{-1/8}+\lstability^{1/3})\\
\le & \mathbb P\left[\delta_3\le t\right] + \tilde O(n^{-1/8}+\lstability^{1/3})\,,
\end{align*}
% Let the vectors in \eqref{eq:normalized_apprx_chain} be $\delta_0$, $\delta_1$, $\delta_2$, $\delta_3$, in the order of appearance.
% Then
% \begin{align*}
% \mathbb P(\max \delta_0\le t)\le & \mathbb P\left[\max\delta_0\le t,~\|\delta_1\|_\infty\le 2\sqrt{\log(n+p)}\right]+2(n+p)^{-1}+\tilde O(n^{-1/8}+\lstability^{1/3})\\
% \le & \mathbb P\left[\max\delta_1\le t+\tilde O(n^{-1/2}+\lstability)\right]+\tilde O(n^{-1/8}+\lstability^{1/3})\\
% \le & \mathbb P\left[\max \delta_2\le t+\tilde O(n^{-1/2}+\lstability)\right] + \tilde O(n^{-1/8}+\lstability^{1/3})\\
% \le & \mathbb P\left[\max \delta_3\le t+\tilde O(n^{-1/2}+\lstability)\right] + \tilde O(n^{-1/8}+\lstability^{1/3})\\
% \le & \mathbb P\left[\max \delta_3\le t\right] + \tilde O(n^{-1/8}+\lstability^{1/3})\,,
% \end{align*}
where the third inequality holds because on $\mathcal E_1$ $|\delta_1-\delta_0|\le \tilde O(n^{-1/2}+\lstability)$; the fourth inequality holds by applying \Cref{thm:1} to the scaled loss functions $\ell_r/\sigma_{rr}^{1/2}$; the fifth inequality holds because when conditioning on the event $\mathcal E$, the two Gaussian vectors have covariance matrices differing by at most $\tilde O(n^{-1/2}+\lstability)$ and applying Theorem 2 of \cite{chernozhukov2015comparison}; the last inequality holds by anti-concentration of Gaussian maxima \citep[][Lemma 2.1]{chernozhukov2013gaussian}.

The corresponding lower probability bound of $\mathbb P(\delta_0\le t)$ can be obtained similarly.
\end{proof}

\subsection{Auxiliary lemmas}

\begin{lemma}[Properties of the difference operator]\label{lem:free_nabla}
	Let $f$, $g$ be two functions of the vector $(X_1,...,X_n,X_1',...,X_n')$ such that for some $j\in [n]$, $\mathbb E_{X_j}g=0$ and $f$ is independent of $X_j'$, then
	$$
\mathbb E [f g] = \mathbb E [(\nabla_j f)g]\,.
	$$
\end{lemma}
\begin{proof}[Proof of \Cref{lem:free_nabla}]
Let $f^j$ be the iid version of $f$ with input $X_j$ replaced by $X_j'$. Now
	It suffices to show that
	$\mathbb E [g f^j] = 0$, which holds true since $\mathbb E[g f^j]=\mathbb E[f^j(\mathbb E_{X_j}g)]=0$.
\end{proof}

\begin{lemma}\label{lem:D} Under \Cref{ass:lstability}, for all $i\in[n]$, $r\in[p]$,
$D_{r,i}$ is $\lstability$-SW.
\end{lemma}
\begin{proof}[Proof of \Cref{lem:D}]
% This lemma has appeared implicitly in \cite{austern2020asymptotics}. We provide a self-contained breif proof.
% Without loss of generality, consider $i=n$ so that $v_i=V$, and we also drop the subindex $r$. By definition
% \begin{align*}
% \mathbb E D_n^2=&\sum_{j,k=1}^{\tilde n} \nabla_i K_j\nabla_i K_k\,.
% \end{align*}

% Consider two case of the summands.
% Case 1: $j=k$.  There are $\tilde n$ such terms. So the total contribution from these terms are
% bounded by $\tilde n (n^{-1/2}\lstability)^2\le \lstability^2$.

% Case 2: $j\neq k$.  There are less than $\tilde n^2$ such terms.
% \begin{align*}
% \mathbb E \nabla_i K_j\nabla_i K_k = & \mathbb E [(\nabla_k \nabla_i K_j)\nabla_i K_k] =\mathbb E [(\nabla_k \nabla_i K_j)(\nabla_j\nabla_i K_k)]\,,
% \end{align*}
% where each equality follows from an application of \Cref{lem:free_nabla}.
% So the total contribution from these summands is bounded by $\tilde n^2 (n^{-1}\lstability)^2 \le \lstability^2$.

% Summing up, we have $\mathbb E D_n^2\le 2 \lstability^2$.
Let $F_{k,i}$ be the sigma field generated by $F_k$ and $X_i'$ for $k\in [n]$, and $F_{0,i}$ be the sigma field genreated by $X_i'$.
Because $\mathbb E(\nabla_i K_j|X_i')=0$ for all $j\neq I_{v_i}$, we have the following martingale sum representation
$$
D_{r,i}=\sum_{k=1}^{n} \mathbb E(D_{r,i}|F_{k,i})-\mathbb E(D_{r,i}|F_{k-1,i})=\sum_{k=1}^n \mathbb E(\tilde\nabla_k D_{r,i}|F_{k,i})\,,
$$
where $\tilde\nabla_k$ is the same operator as $\nabla_k$ except that it replaces $X_k$ by $X_k''$, a further iid copy.  This is to make sure that the difference operator $\tilde\nabla_i$ does not interfere with $X_i'$, which is already involved in $D_{r,i}$.

For each $k\in[n]$, if $k\in\{i,j\}$, we have that
$\tilde\nabla_k\nabla_i K_j$ is $(\lstability n^{-1/2},\alpha)$-SW by part 1 of \Cref{ass:lstability} and closure of sub-Weibull tails under additions.

For each $k\in[n]\backslash\{i,j\}$, we have that
$\tilde\nabla_k\nabla_i K_j$ is either $0$ (if $k\in I_{v_j}$) or $(\lstability n^{-3/2},\alpha)$-SW (if $k\neq i,j$, by part 2 of \Cref{ass:lstability}).

So overall, we conclude that $\tilde \nabla_k D_{r,i}$ is $(\lstability n^{1/2},\alpha)$-SW.  The claimed result follows from \Cref{lem:MG-subW}.
\end{proof}

\begin{lemma}\label{lem:cov_concentration}
	Under \Cref{ass:lstability}, $\sigma_{rs,1}-\sigma_{rs}$ is $\lstability$-SW.
\end{lemma}
\begin{proof}[Proof of \Cref{lem:cov_concentration}]
Without loss of generality, we work with $v=1$. First write $\sigma_{rs,1}-\sigma_{rs}$ as the sum of
martingale increments
	\begin{align}
	\sigma_{rs,1}-\sigma_{rs} = \sum_{i=1}^{\tilde n}\mathbb E(\nabla_i \sigma_{rs,1}|F_i)\,.\label{eq:mg_decomp_sigmav}
	\end{align}
Next we control each $\mathbb E(\nabla_i \sigma_{rs,1}|F_i)$.
Let $\|\cdot\|$ be any $L_q$ norm with  $q\ge 1$,
\begin{align*}
	&\|\mathbb E(\nabla_i \sigma_{rs,1}|F_i)\|\le  \|\nabla_i \sigma_{rs,1}\|\\
\le & \left\|\mathbb E_0\left( K_r(X_0,\mathbf X_{-1})K_s(X_0,\mathbf X_{-1})-K_r(X_0,\mathbf X_{-1}^i)K_s(X_0,\mathbf X_{-1}^i)\right)\right\|\\
\le &\left\|K_r(X_0,\mathbf X_{-1})K_s(X_0,\mathbf X_{-1})-K_r(X_0,\mathbf X_{-1}^i)K_s(X_0,\mathbf X_{-1}^i)\right\|\\
\le & \left\|\left[\nabla_i K(X_0,\mathbf X_{-1})\right] K_s(X_0,\mathbf X_{-1})\right\|+\left\|K_r(X_0,\mathbf X_{-1}^i)\left[\nabla_i K_s(X_0,\mathbf X_{-1})\right] \right\|\,.
\end{align*}
Then it follows from \Cref{ass:lstability} and \Cref{pro:basic_SW} that 
$\mathbb E(\nabla_i \sigma_{rs,1}|F_i)$ is $\lstability n^{-1/2}$-SW.
Further applying \Cref{lem:MG-subW} to the martingale sum \eqref{eq:mg_decomp_sigmav} we conclude that
$\sigma_{rs,1}-\sigma_{rs}$ is $\lstability$-SW.
\end{proof}

% \begin{lemma}\label{lem:sigma_sw}
% 	Under \Cref{ass:sw}, $\sigma_{r,v}^2$ is sub-Weibull.
% \end{lemma}
% \begin{proof}[Proof of \Cref{lem:sigma_sw}]
% By symmetry, we only need to prove for $v=1$.
% By construction, $\sigma_{r,1}^2 = \mathbb E[K_{r,1}^2|\mathbf X_{-1}]$. Therefore,
% $\sigma_{r,1}^2$ has a tail that is no heavier than $K_{r,1}^2$ as $\|\sigma_{r,1}^2\|_q\le \|K_{r,1}^2\|_q$ for all $q\ge 1$. The desired result follows immediately from the assumption that $K_{r,i}$ is sub-Weibull.
% \end{proof}

\begin{lemma}[Bridging between smooth function and CDF of maximum]\label{lem:h_function}
For any $\beta>0$, there exists a function $h=h_{\beta}:\mathbb R^p\mapsto \mathbb R$, such that for any random vector $\mathbf Z\in\mathbb R^p$,
$$
\mathbb P\left(\max_{r} Z_r\le t\right) \le \mathbb E h(\mathbf Z)\le \mathbb P\left(\max_{r} Z_r\le t+\frac{\log p+1}{\beta}\right)
$$
and, for some universal constant $C$,
\begin{align*}
M_2(h)=&\sup_{z\in\mathbb R^p}\sum_{r,s=1}^p|\partial_r\partial_s h(z)|\le C \beta^2\,,\\
M_3(h)=&\sup_{z\in\mathbb R^p}\sum_{r,s,u=1}^p|\partial_r\partial_s \partial_u h(z)|\le C \beta^3\,.
\end{align*}
\end{lemma}
\begin{proof}[Proof of \Cref{lem:h_function}]
See Lemma A.5 and Corollary I.1 of \cite{chernozhukov2013gaussian}.	
\end{proof}

\section{Proofs for stability condition examples}
%!TEX root = ./hd_cv_clt.tex

\subsection{Proof of \Cref{prop:sgd1}}

\Cref{prop:sgd1} is a direct consequence of \Cref{lem:first_order_err}. 

\begin{lemma}[Strongly convex first order error]
  \label{lem:first_order_err}
  For $t\ge i$,
  $$\|\nabla_i\hat\theta_t\|\le \frac{2L}{\beta}i^{-a}\exp\left[-c_{a,\beta,\gamma}((t+1)^{1-a}-(i+1)^{1-a})\right]\,,$$
  where $c_{a,\beta,\gamma}=2\gamma/[(1-a)(\beta+\gamma)]$.
\end{lemma}
\begin{proof}
    By construction, for each $t\ge 1$, $\alpha_t$ satisfies
$$
\alpha_t\le \frac{2}{\beta+\gamma}
$$
and according to Lemma 3.7 of \cite{hardt2016train} then the SGD update at step $t$ satisfies  $\|G_{\psi,\alpha_t}(\theta,x)-G_{\psi,\alpha_t}(\theta',x)\|\le (1-\frac{2\beta\gamma}{\beta+\gamma}\alpha_t)\|\theta-\theta'\|$ for parameter update function $G:\Theta \rightarrow \Theta$.

Then for any $t\ge i$,
\begin{align*}
\|\nabla_i\hat\theta_t\|\le & \frac{2L}{\beta} i^{-a} \prod_{k=j+1}^t\left(1-\frac{2\beta\gamma}{\beta+\gamma}\alpha_k\right)  \\
 = & \frac{2L}{\beta} i^{-a}\exp\left(\sum_{k=i+1}^t \log\left(1-\frac{2\beta\gamma}{\beta+\gamma}\alpha_k\right)\right)\\
 \le & \frac{2L}{\beta} i^{-a}\exp\left(-\frac{2\beta\gamma}{\beta+\gamma}\sum_{k=i+1}^t \alpha_k\right)\\
 = &\frac{2L}{\beta} i^{-a}\exp\left(-\frac{2\gamma}{\beta+\gamma}\sum_{k=i+1}^t k^{-a}\right)\\
 \le &\frac{2L}{\beta} i^{-a}\exp\left(-\frac{2\gamma}{\beta+\gamma}\int_{j+1}^{t+1} x^{-a}dx\right)\\
 = &\frac{2L}{\beta} i^{-a}\exp\left(-\frac{2\gamma}{(1-a)(\beta+\gamma)}\left((t+1)^{1-a}-(i+1)^{1-a}\right)\right)\,.\qedhere
\end{align*}
\end{proof}

\begin{proof}[Proof of \Cref{prop:sgd1}]
First we apply \Cref{lem:first_order_err}.

When $i> n/2$, we have $\|\nabla_i\hat\theta_n\|\le \frac{2L}{\beta}(n/2)^{-a}$.

 When $i\le n/2$, we have $\|\nabla_i\hat\theta_n\|\lesssim \frac{L}{\beta}\exp(-c_a \frac{\gamma}{\beta}n^{1-a})$
with $c_a=\frac{1-2^{-(1-a)}}{1-a}$.

So we conclude that
$$
\|\nabla_i\hat\theta\|\lesssim \frac{L}{\beta}n^{-a}
$$
whenever $\frac{\gamma}{\beta}\ge \frac{a(1-a)}{1-2^{-(1-a)}}\frac{\log n}{n^{(1-a)}}$.
\end{proof}

\subsection{Proof of \Cref{prop:sgd2}}
\begin{proof}
Without loss of generality, we assume that the data points arrive in their natural indexed order. Let $\hat \theta = \hat \theta(\mathbf X)$ be the solution using the entire original data sequence  and denote $\hat \theta^i = \hat \theta(\mathbf X^i)$ and $\hat \theta^{ij} = \hat \theta(\mathbf X^{i,j})$, where $\mathbf X^{i,j}$ is the vector obtained by replacing $X_j$ in $\mathbf X^i$ with its iid copy $X^\prime_j$.
% Denote the following
% \begin{itemize}
%   \item $\hat\theta$: the solution from the original data sequence
%   \item $\hat\theta^i$: the solution with $X_i'$
%   \item $\hat\theta^j$: the solution with $X_j'$
%   \item $\hat\theta^{ij}$: the solution with $X_i',~X_j'$  
% \end{itemize}

Then
\begin{align*}
\left\|\nabla_i\nabla_j\hat\theta_{j}\right\|= & \left|\hat\theta_{j} - \hat\theta^i_{j} - \hat\theta^j_{j}+\hat\theta^{ij}_{j}\right|\\
=&\big\|\hat\theta_{j-1}-\alpha_{j}\dot\psi(\hat\theta_{j-1};X_j) -
\hat\theta^i_{j-1}+\alpha_{j}\dot\psi(\hat\theta^i_{j-1};X_j) \\
&-\hat\theta_{j-1}+\alpha_{j}\dot\psi(\hat\theta_{j-1};X_j') +
\hat\theta^i_{j-1}-\alpha_{j}\dot\psi(\hat\theta^i_{j-1};X_j')\big\| \\
=&\alpha_{j}\left\|\dot\psi(\hat\theta_{j-1};X_j)-\dot\psi(\hat\theta^i_{j-1};X_j)-\dot\psi(\hat\theta_{j-1};X_j')+\dot\psi(\hat\theta^i_{j-1};X_j')\right\|\\
\le & 2\beta\alpha_{j}  \|\nabla_i\hat\theta_{j-1}\|.
\end{align*}
For $t>j$,
\begin{align*}
  \nabla_i\nabla_j\hat\theta_{t}= & \hat\theta_{t} - \hat\theta^i_{t} - \hat\theta^j_{t}+\hat\theta^{ij}_{t}\\
=&\hat\theta_{t-1}-\alpha_{t}\dot\psi(\hat\theta_{t};X_t) -
\hat\theta^i_{t-1}+\alpha_{t}\dot\psi(\hat\theta^i_{t-1};X_t) \\
&-\hat\theta^j_{t-1}+\alpha_{t}\dot\psi(\hat\theta^j_{t-1};X_t) +
\hat\theta^{ij}_{t-1}-\alpha_{t}\dot\psi(\hat\theta^{ij}_{t-1};X_t) \\
=&\nabla_i\nabla_j \hat\theta_{t-1} - \alpha_t\left[\dot\psi(\hat\theta_{t-1},X_t)-\dot\psi(\hat\theta_{t-1}^i,X_t)-\dot\psi(\hat\theta_{t-1}^j,X_t)+\dot\psi(\hat\theta_{t-1}^{ij},X_t)\right]\,.
\end{align*}
Using the mean value theorem, there exist $\tilde\theta_{t-1}\in[\hat\theta_{t-1},\hat\theta_{t-1}^i]$ and $\tilde\theta_{t-1}'\in[\hat\theta_{t-1}^j,\hat\theta_{t-1}^{ij}]$ such that
\begin{align*}
&  \dot\psi(\hat\theta_{t-1},X_t)-\dot\psi(\hat\theta_{t-1}^i,X_t)-\dot\psi(\hat\theta_{t-1}^j,X_t)+\dot\psi(\hat\theta_{t-1}^{ij},X_t)\\
=& \ddot\psi(\tilde \theta_{t-1},X_t)(\hat\theta_{t-1}-\hat\theta_{t-1}^i)-\ddot\psi(\tilde\theta_{t-1}',X_t)(\hat\theta_{t-1}^j-\hat\theta_{t-1}^{ij})\\
=&\ddot\psi(\tilde \theta_{t-1},X_t) \nabla_i\nabla_j\hat\theta_{t-1}+\left[\ddot\psi(\tilde\theta_{t-1},X_t)-\ddot\psi(\tilde\theta_{t-1}',X_t)\right](\hat\theta_{t-1}^j-\hat\theta_{t-1}^{ij})
\end{align*}

Plugging this to the RHS of the previous equation, we get
\begin{align*}
  \nabla_i\nabla_j\hat\theta_t = \left[I-\alpha_t \ddot\psi(\tilde \theta_{t-1},X_t)\right] \nabla_i\nabla_j\hat\theta_{t-1}+\alpha_t\left[\ddot\psi(\tilde\theta_{t-1},X_t)-\ddot\psi(\tilde\theta_{t-1}',X_t)\right](\hat\theta_{t-1}^j-\hat\theta_{t-1}^{ij})\,.
\end{align*}

Let $\eta_t=\|\nabla_i\nabla_j\hat\theta_t\|$, we get
\begin{align*}
\eta_t\le (1-\alpha_t\gamma)\eta_{t-1}+2L_2 \alpha_t\|\tilde \theta_{t-1}-\tilde\theta_{t-1}'\|\cdot\|(\hat\theta_{t-1}^j-\hat\theta_{t-1}^{ij})\|,
\end{align*}
where the second term on the RHS uses Lipschitz property of $\ddot\psi$ (with Lipschitz constant $L_2$).

To control $\tilde\theta_{t-1}-\tilde\theta_{t-1}'$, we have
\begin{align*}
  \|\tilde\theta_{t-1}-\tilde\theta_{t-1}'\|\le & \|\hat\theta_{t-1}-\hat\theta_{t-1}^j\|+\|\hat\theta_{t-1}-\hat\theta_{t-1}^{ij}\|+\|\hat\theta_{t-1}^i-\hat\theta_{t-1}^j\|+\|\hat\theta_{t-1}^i-\hat\theta_{t-1}^{ij}\|\\
\le & 4 \frac{2L}{\beta}j^{-a}\exp\left[-c_{a,\beta,\gamma}(t^{1-a}-(j+1)^{1-a})\right]\\
&~~+
2\frac{2L}{\beta}i^{-a}\exp\left[-c_{a,\beta,\gamma}(t^{1-a}-(i+1)^{1-a})\right]\,,
\end{align*}
by \Cref{lem:first_order_err} directly on $\hat\theta_{t-1}-\hat\theta_{t-1}^i$ and $\hat\theta_{t-1}^i-\hat\theta_{t-1}^{ij}$, and also both $\hat\theta_{t-1}-\hat\theta_{t-1}^{ij}=\hat\theta_{t-1}-\hat\theta_{t-1}^i+\hat\theta_{t-1}^i-\hat\theta_{t-1}^{ij}$ and $\hat\theta^i_{t-1}-\hat\theta_{t-1}^{j} = \hat\theta_{t-1}^i-\hat\theta_{t-1}^{ij}+\hat\theta_{t-1}^{ij}-\hat\theta_{t-1}^j$.

Let (for $t\ge j+1$)
\begin{align*}
R_{i,j,t}=&\frac{16L^2 L_2}{\beta^3}t^{-a}i^{-a}\exp\left[-c(t^{1-a}-(i+1)^{1-a})\right]j^{-a}\exp\left[-c(t^{1-a}-(j+1)^{1-a})\right]\\
&+~~\frac{8L^2 L_2}{\beta^3}t^{-a}i^{-2a}\exp\left[-2c(t^{1-a}-(i+1)^{1-a})\right].
\end{align*}
Therefore, we establish the recursive inequality,
$$
\eta_{t}\le (1-\alpha_{t}\gamma)\eta_{t-1}+R_{i,j,t}\,,
$$
for $t\ge j+1$.

We now seek to upperbound $\eta_t$. To do so, let $c'=\gamma/[(1-a)\beta]\le c_{a,\beta,\gamma}$. Thus
\begin{align}
\eta_n \le & (1-\alpha_{n}\gamma)\eta_{n-1}+R_{i,j,n}\nonumber\\
\le & (1-\alpha_{n}\gamma)\left[(1-\alpha_{n-1}\gamma)\eta_{n-2}+R_{i,j,n-1}\right]+R_{i,j,n}\nonumber\\
% \le & \cdots\\
\le & \left[\prod_{k=j+1}^n(1-\alpha_j\gamma)\right] \eta_j + \sum_{k=j+1}^{n} \left[\prod_{l=k+1}^n(1-\alpha_l\gamma)\right] R_{i,j,k}\nonumber\\
\begin{split}\label{eqn:etanbound}
\le & \exp\left[-c'((n+1)^{1-a}-(j+1)^{1-a})\right]\eta_j\\
&+ ~~\sum_{k=j+1}^{n} \exp\left[-c'((n+1)^{1-a}-(k+1)^{1-a})\right] R_{i,j,k}\;.
\end{split}
\end{align}
The final line appeals to the fact that for positive integers $j<t$, $\prod_{j+1}^t (1-\alpha_j \gamma)\le \exp\left[-\frac{\gamma}{(1-a)\beta}\left((t+1)^{1-a}-(j+1)^{1-a}\right)\right]$.

There are two terms in $R_{i,j,k}$, the contribution from the first term in the sum term of \eqref{eqn:etanbound} is
\begin{align*}
 &\frac{16 L^2 L_2}{\beta^3}i^{-a}j^{-a} \sum_{k=j+1}^n k^{-a} \exp\left[-c'((n+1)^{1-a}-(k+1)^{1-a})\right]\\
&\times
 \exp\left[-c'(k^{1-a}-(j+1)^{1-a})\right]\exp\left[-c'(k^{1-a}-(i+1)^{1-a})\right]\\
 \le & \frac{16 L^2 L_2}{\beta^3}i^{-a}j^{-a}\exp\left[-c'((n+1)^{1-a}-(i+1)^{1-a})\right] \\
 &\times \exp\left[2c'(2^{1-a}-1)\right] \sum_{k=j+1}^n k^{-a}\\
 \le & \frac{16 L^2 L_2 e^{2}}{\beta^3}i^{-a}j^{-a}\exp\left[-c'((n+1)^{1-a}-(i+1)^{1-a})\right] \sum_{k=j+1}^n k^{-a}\\
 \le &\frac{16 L^2 L_2e^{2}}{(1-a)\beta^3}i^{-a}j^{-a}\exp\left[-c'((n+1)^{1-a}-(i+1)^{1-a})\right](n^{1-a}-j^{1-a})
\end{align*}
the first inequalty uses the facts that $k\ge j+1$ and that $\exp(c'(k+1)^{1-a}-c'k^{1-a})\le \exp((2^{1-a}-1)c')$ since $k\ge 1$ and $1-a\in(0,1)$.

Let $c_2$ be a constant depending only on $a$.
When $i\le n-c_2 n^a \log n$, we have
\begin{align*}
  (n+1)^{1-a}-(i+1)^{1-a} = & (n+1)^{1-a}\left[1-\left(1-\frac{n-i}{n+1}\right)^{1-a}\right]\\
  \ge & (n+1)^{1-a}(1-a)\frac{n-i}{n+1}\\
  \ge & c_2(1-a) \log n\,,
\end{align*}
where the second lines uses $(1-x)^b\le 1-bx$ for $x\in[0,1)$ and $b\in(0,1)$. 

When $c_2\ge (3-a)/[c'(1-a)]$,
\begin{align*}
    & \frac{16 L^2 L_2e^{2}}{(1-a)\beta^3}i^{-a}j^{-a}\exp\left[-c'((n+1)^{1-a}-(i+1)^{1-a})\right](n^{1-a}-j^{1-a}) \\
    \lesssim &  \frac{L^2L_2}{(1-a)\beta^3}n^{-2} \\
    \leq & \frac{L^2L_2}{(1-a)\beta^3}n^{-2a}\log n \,.
\end{align*}
for large $n$.

When $i\ge n-c_2 n^a\log n\ge n/2$, then total number of terms in the sum over $k$ is at most $c_2 n^a\log n$, but
$i,j$ are both lower bounded by $n/2$. So 
\begin{align*}
     & \frac{16 L^2 L_2 e^{2}}{\beta^3}i^{-a}j^{-a}\exp\left[-c'((n+1)^{1-a}-(i+1)^{1-a})\right] \sum_{k=j+1}^n k^{-a}\\
    \lesssim & \frac{L^2 L_2}{\beta^3} n^{-2a} (c_2 n^a \log n )n^{-a}\\
    \leq & \frac{L^2 L_2}{(1-a)\beta^3} c_2 n^{-2a}  \log n 
\end{align*}

For the second term in $R_{i,j,k}$, the sum in \eqref{eqn:etanbound} simplifies to
\begin{align*}
  &\frac{8 L^2 L_2}{\beta^3}i^{-2a} \sum_{k=j+1}^n k^{-a} \exp\left[-c'((n+1)^{1-a}-(k+1)^{1-a})\right]
 \exp\left[-2c'(k^{1-a}-(i+1)^{1-a})\right]\\
 \lesssim & \frac{8 L^2 L_2}{\beta^3}i^{-2a} \sum_{k=j+1}^n k^{-a} \exp\left[-c'((n+1)^{1-a}-(i+1)^{1-a})\right]\\
 \lesssim & \frac{L^2L_2}{(1-a)\beta^3}c_2 n^{-2a}\log n\,.
\end{align*}
where the final line uses the same argument as the first term. Therefore, we have the following bound  
$$
\sum_{k=j+1}^{n} \exp\left[-c'((n+1)^{1-a}-(k+1)^{1-a})\right] R_{i,j,k} \lesssim \frac{L^2L_2}{(1-a)\beta^3}c_2 n^{-2a}\log n\,.$$

Lastly, to finish bounding \eqref{eqn:etanbound} the contribution from $\eta_j$ is
\begin{align*}
  &\exp\left[-c'((n+1)^{1-a}-(j+1)^{1-a})\right]\eta_j\\
\le & 2\beta\alpha_j \|\nabla_i\hat\theta_{j-1}\|\exp\left[-c'((n+1)^{1-a}-(j+1)^{1-a})\right]\\
\le & \frac{4L}{\beta}j^{-a}i^{-a}\exp\left[-c'\left(j^{1-a}-(i+1)^{1-a}\right)\right]\exp\left[-c'((n+1)^{1-a}-(j+1)^{1-a})\right]\\
\lesssim & \frac{4L}{\beta}j^{-a}i^{-a}\exp\left[-c'\left((n+1)^{1-a}-(i+1)^{1-a}\right)\right] \,.
\end{align*}
When $ i,j < n/2$,
\begin{align*}
	\exp\left[-c'\left((n+1)^{1-a}-(i+1)^{1-a}\right)\right] & \lesssim \exp\left[-c'\left( 1-2^{a-1}\right) n^{1-a}\right] \\
	& \lesssim n^{-2a} \log n
\end{align*}
for large $n$. When $i,j \geq n/2$, 
\begin{align*}
	\exp\left[-c'\left((n+1)^{1-a}-(i+1)^{1-a}\right)\right] & \leq \left(\frac{n}{2}\right)^{-2a} \\
	& \lesssim n^{-2a} \log n
\end{align*}
for large $n$. Therefore, we have shown
\begin{align*}
\left\Vert \nabla_j \nabla_i \hat\theta_n \right\Vert & \lesssim  c(L, L_2, a, \beta, \gamma)n^{-2a} \log n\,. \qedhere
\end{align*}
\end{proof}

\subsection{Proof of \Cref{prop:lossstab}}
\begin{proof}
Suppose $Y_i=\sum_{j=1}^\infty \beta_j Z_{ij} +\epsilon_i$ with ${\rm Var}(\epsilon_i)=\sigma_\epsilon^2$. Let $1<J_r<J_s$ be two integers and
$\hat f_r(Z_i)=\sum_{j=1}^{J_r}\hat\beta_j Z_{ij}$, and $\hat f_s(Z_i)=\sum_{j=1}^{J_s}\hat\beta_j Z_{ij}$, where
$$
\hat\beta_j=\frac{1}{n}\sum_{i=1}^n Y_i Z_i = \beta_j+\delta_j
$$
with
$$
\delta_j =\frac{1}{n}\sum_{i=1}^n Y_i Z_i-\beta_j = \frac{1}{n}\sum_{i=1}^n\left[\epsilon_iZ_i+f(Z_i)Z_i-\beta_j\right]\,,
$$
which satisfies
$\mathbb E\delta_j=0$ and ${\rm Var}(\delta_j)=\sigma_j^2/n$ for some $\sigma_j^2\in [\sigma_\epsilon^2,\sigma_\epsilon^2+\|f\|_\infty^2]$.

Consider the difference for loss functions,
\begin{align*}
T= &(Y_0-\hat f_r(Z_0))^2-(Y_0-\hat f_s(Z_0))^2\\
 =& 2\epsilon_0(\hat f_s(Z_0)-\hat f_r(Z_0)) +\left[2f(Z_0)-\hat f_r(Z_0)-\hat f_s(Z_0)\right]\left[\hat f_s(Z_0)-\hat f_r(Z_0)\right]\,.
\end{align*}

We are interested in the standard deviation of $T$. Let $\mathbf X=(Z_i,Y_i)_{i=1}^n$,
\begin{align*}
  {\rm Var}(T) = \mathbb E\left[{\rm Var}(T|\mathbf X, Z_0)\right] + {\rm Var}\left[\mathbb E(T|\mathbf X, Z_0)\right]\,.
\end{align*}

For the first term
\begin{align*}
  &\mathbb E\left[{\rm Var}(T|\mathbf X, Z_0)\right]\\
   =& 4\sigma_\epsilon^2\mathbb E\left[\hat f_s(Z_0)-\hat f_r(Z_0)\right]^2\\
 =&4\sigma_\epsilon^2\mathbb E\left[\sum_{j=J_r+1}^{J_s}\hat\beta_j Z_{0j}\right]^2\\
 =&4\sigma_\epsilon^2\mathbb E\sum_{j,k=J_r+1}^{J_s}\hat\beta_j Z_{0j}\hat\beta_k Z_{0k}\\
 =&4\sigma_\epsilon^2\mathbb E\sum_{j=J_r+1}^{J_s}\hat\beta_j^2 Z_{0j}^2\\
 \asymp&4\sigma_\epsilon^2\sum_{j=J_r+1}^{J_s}\left(\beta_j^2+\frac{\sigma_j^2}{n}\right)\\
 \ge & 4\sigma_\epsilon^2\left(\sum_{j=J_r+1}^{J_s} \beta_j^2+\frac{J_s-J_r}{n}\sigma_\epsilon^2\right)\,.
\end{align*}
If we regard $\sigma_\epsilon$ as a positive constant, then
$$
\sigma_{ss}^{(r)} \gtrsim \left(\sum_{j=J_r+1}^{J_s}\beta_j^2\right)^{1/2} + \frac{\sqrt{J_s-J_r}}{\sqrt{n}}\,.
$$

Now taking the $\nabla_i$ operator. By construction, $T'$ is the version of $T$ obtained by replacing $(Z_i,Y_i)$ by an iid copy $(Z_i',Y_i')$, keeping everything else unchanged. To derive a simple expression of $\nabla_i T = T'-T$, we first expand $T$
\begin{align*}
  T = & 2\epsilon_0\sum_{j=J_r+1}^{J_s}\hat\beta_j Z_{0j}\\
  &+\left[-2\sum_{j=1}^{J_r}\delta_j Z_{0j}+\sum_{j=J_r+1}^{J_s}(\beta_j-\delta_j) Z_{0j}+2\sum_{j=J_s+1}^\infty \beta_j Z_{0j}\right]\sum_{j=J_r+1}^{J_s}\hat\beta_j Z_{0j} \\
  \nabla_i T = & 2\epsilon_0\sum_{j=J_r+1}^{J_s}(\nabla_i \delta_j) Z_{0j}\\
 &+\left[-2\sum_{j=1}^{J_r}\nabla_i\delta_j Z_{0j}-\sum_{j=J_r+1}^{J_s}\nabla_i\delta_j Z_{0j}\right]\sum_{j=J_r+1}^{J_s}\hat\beta_j Z_{0j}\\
 &+\left[-2\sum_{j=1}^{J_r}\delta_j Z_{0j}+\sum_{j=J_r+1}^{J_s}(\beta_j-\delta_j) Z_{0j}+2\sum_{j=J_s+1}^\infty \beta_j Z_{0j}\right]\sum_{j=J_r+1}^{J_s}\nabla_i\delta_j Z_{0j}\\
 =& A_1+A_2+A_3\,,
\end{align*}
where $\nabla_i\delta_j=n^{-1}(Y_i' Z'_{ij}-Y_i Z_{ij})$. 

Since $\epsilon_0$ and $Z_{0j}$ are assumed sub-Weibull and $\mathbb E (Z_{0j} | Z_{01}, \dots,Z_{0,j-1})$ for all $j$, we can use \Cref{lem:MG-subW} to obtain that
$$ A_1 \text{ is } \left(\frac{\sqrt{J_s}}{n}\right)\text{-SW}$$
\begin{align*}
   A_2 \text{ is } \frac{\sqrt{J_s}}{n}\left(\|\beta_{J_r+1,J_s}\|+\frac{\sqrt{J_s}}{\sqrt{n}}\right)\text{-SW}
\end{align*}
\begin{align*}
   A_3 \text{ is } \left(\frac{\sqrt{J_s}}{\sqrt{n}}+\|\beta_{J_r+1,\infty}\|\right)\frac{\sqrt{J_s}}{n}\text{-SW}
\end{align*}
where
$$
\|\beta_{a,b}\|^2=\sum_{j=a}^b \beta_j^2\,.
$$

Since $ \beta_j\asymp j^{-\frac{1+a}{2}} $ for some $a>0$,
$$
\|\beta_{r,\infty}\| \asymp \left(\int_{J_r}^\infty x^{-1-a}dx\right)^{1/2}\asymp J_r^{-a/2}\,,
$$
and 
$$\|\beta_{J_r+1,J_s}\|\asymp J_r^{-a/2}$$
provided that $J_s\ge c J_r$ for some constant $c>1$. Therefore
$$
\sigma_{ss}^{(r)} \gtrsim J_r^{-a/2} + \frac{\sqrt{J_s}}{\sqrt{n}}\,,
$$
and first order stability condition reduces to
\begin{align*}
  \frac{\sqrt{J_s}}{n}+\frac{\sqrt{J_s}}{n}\left(\frac{\sqrt{J_s}}{\sqrt{n}}+J_r^{-a/2}\right)\ll \frac{1}{\sqrt{n}}\left(J_r^{-a/2}+\frac{\sqrt{J_s}}{\sqrt{n}}\right)
\end{align*}
which requires
\begin{align*}
  &J_s J_r^{a/2}\ll n.
\end{align*} 

Now for second order stability.  For $k\neq i$,
\begin{align*}
  \nabla_k(\nabla_i T) = & \left[-2\sum_{j=1}^{J_r}\nabla_i\delta_j Z_{0j}-\sum_{j=J_r+1}^{J_s}\nabla_i\delta_j Z_{0j}\right]\sum_{j=J_r+1}^{J_s}\nabla_k \delta_j Z_{0j}\\
 &+\left[-2\sum_{j=1}^{J_r}\nabla_k\delta_j Z_{0j}-\sum_{j=J_r+1}^{J_s}\nabla_k\delta_j Z_{0j}\right]\sum_{j=J_r+1}^{J_s}\nabla_i\delta_j Z_{0j}\\
 =& B_1+B_2\,.
\end{align*}
Again, using \Cref{lem:MG-subW}, then $B_1$ and $B_2$ are $\frac{J_s}{n^2}$-SW. Then the second order stability reduces to
$$
\frac{J_s}{n^2}\ll n^{-3/2}\left(J_r^{-a/2}+\frac{\sqrt{J_s}}{\sqrt{n}}\right)
$$
which is equivalent to
\begin{align*}
J_sJ_r^{a/2}\ll n^{1/2}\,.&\qedhere
\end{align*}
\end{proof}

\section{Proof for deterministic centering}
%!TEX root = ./hd_cv_clt.tex
\subsection{Preparation}
In preparation for the proof, we first take a closer look at some intermediate quantities involved in forming the asymptotic covariance term.  Intuitively, the variance contributed by $X_i$ in the term $\ell_r(X_j;\mathbf X_{-v_j})$ for $j\notin I_{v_i}$ is from the variability of $R_r(\mathbf X_{-v_j})$.  Based on this intuition, we can reduce \eqref{eq:Xi_in_Rcv} to
\begin{align}\label{eq:g}
	\ell_r(X_i;\mathbf X_{-v_i})+\sum_{j\notin I_{v_i}}R_r(\mathbf X_{-v_j})\eqqcolon g_{r,i}\,.
\end{align}
The quantities in \eqref{eq:g} still involve many sample points $X_j$ ($j\neq i$). In order to pinpoint the variance contributed by $X_i$ alone, we consider the following difference versions of $g_{r,i}$:
$$
\mathbb E (g_{r,i}|F_i)-\mathbb E(g_{r,i}|F_{i-1})\,,\text{ and } g_{r,i}-\mathbb E(g_{r,i}|\mathbf X^{-i})\,,
$$
where $F_i$ is the $\sigma$-field generated by $(X_1,...,X_i)$ and $\mathbf X^{-i}=(X_1,...,X_{i-1},X_{i+1},...,X_n)$.  The reason to consider these two differences is rather technical, where the former allows us to express $g_{r,i}$ as the sum of a sequence of martingale increments $\{\mathbb E (g_{r,i}|F_i)-\mathbb E(g_{r,i}|F_{i-1}):~i=1,...,n\}$, and the latter provides $\mathbb E(g_{r,i}|\mathbf X^{-i})$ as a leave-one-out approximation to $g_{r,i}$ with a manageable difference.

Let
\begin{equation}\label{eq:C_rsi}
C_{rs,i} = \left[\mathbb E (g_{r,i}|F_i)-\mathbb E(g_{r,i}|F_{i-1})\right]\left[ g_{s,i}-\mathbb E(g_{s,i}|\mathbf X^{-i})\right]\,.
\end{equation}
In \Cref{lem:E(C_rsi)-phi_rs} we will show that
\begin{equation}\label{eq:det_var_approx}
\mathbb E C_{rs,i} \approx \phi_{rs}\,.
\end{equation}
A key step in the proof is to ensure that this covariance is indeed contributed mostly by $X_i$, which amounts to controlling
$$
\mathbb E (C_{rs,i}|\mathbf X^{-i})-\mathbb E(C_{rs,i})\,.
$$
It can be shown that $\|\mathbb E (C_{rs,i}|\mathbf X^{-i})-\mathbb E(C_{rs,i})\|_2$ is small using the standard Efron-Stein inequality.  However, the high-dimensionality requires some uniform bound of the realized values $\mathbb E (C_{rs,i}|\mathbf X^{-i})-\mathbb E(C_{rs,i})$ over the triplet $(r,s,i)$. This is established using our sub-Weibull conditions in \Cref{lem:C_rsi}.
% \begin{assumption}\label{ass:C_rsi_SW}
% 	For all $(r,s,i)\in [p]^2\times[n]$ and $j\notin I_{v_i}$, $\frac{\mathbb E (C_{rs,i}|\mathbf X^{-i})-\mathbb E(C_{rs,i})}{\|\mathbb E (C_{rs,i}|\mathbf X^{-i})-\mathbb E(C_{rs,i})\|_2}$ and $\frac{\nabla_i R_{r}(\mathbf X_{-v_j})}{\|\nabla_i R_{r}(\mathbf X_{-v_j})\|_2}$ are sub-Weibull with constants independent of $(n,p)$.
% \end{assumption}
% \begin{remark} The first part of \Cref{ass:C_rsi_SW} essentially assumes that the covariance in \eqref{eq:det_var_approx} comes mainly from $X_i$ and the residual
% term does not exceed its $L_2$ norm by too much.  It can be viewed as the counterpart of \Cref{ass:var_stability} in the random centering and scaling case.  The second part of \Cref{ass:C_rsi_SW} is analogous to \Cref{ass:sw}.
% \end{remark}

\subsection{Proof of \Cref{thm:thm3}}
\begin{proof}[Proof of \Cref{thm:thm3}]
Throughout this proof, the notation may be different from that in the proof of \Cref{thm:1}, especially for $W$ and $Z$, due to the different centering and scaling. 

Let  $\mathbf Y=(Y_1,...,Y_p)=\sum_{i=1}^n n^{-1/2} \varepsilon_i$, where $\varepsilon_i\stackrel{iid}{\sim}N(0,\boldsymbol{\Phi})$.  Define $\boldsymbol{\varepsilon}^i$ as the vector $(\varepsilon_i:1\le i\le n)$ with the $i$th element $\varepsilon_i$ replaced by its iid copy $\varepsilon_i'$. Note that each $\varepsilon_i=(\varepsilon_{1,i},...,\varepsilon_{p,i})$ is itself a $p$-dimensional vector.  %Let $\mathbf Y^i$ be the corresponding vector with constructed from $\boldsymbol{\varepsilon}^i$.

Let $F_i$ be the $\sigma$-field generated by $\{(X_j):j\le i\}$, and for any function $f$ acting on $\mathbf X$, define
$$
\nabla_i f =  f(\mathbf X)-f(\mathbf X^i)
$$
and
\begin{equation}\label{eq:delta2nabla}
\Delta_i f = \mathbb E(f|F_i)-\mathbb E(f|F_{i-1})=\mathbb E(\nabla_i f|F_i)\,.
\end{equation}

Define quantities
\begin{align*}
	W_r  &= \frac{1}{\sqrt{n}} \sum_{i=1}^n \left\{ \ell_r(X_i, \mathbf X_{-v_i}) - R_r \right\}\,, \\
	W_r^i  &= \mathbb E(W_r | \mathbf X^{-i})\,.
%	g_{r,i}  & = \ell_r(X_i, \mathbf X_{-v_i}) + \sum_{j \notin I_{v_i}} R_r(\mathbf X_{-v_j})\,,
	\end{align*}
where $W_r$ is the deterministically centered quantity for which we would like to establish a Gaussian comparison, and $W_r^i$ is the corresponding leave-one-out version.

We then have the following useful facts
\begin{align}
	W_r-W_r^i=&W_r-\mathbb E(W_r(\mathbf X^i)|\mathbf X)=\mathbb E(\nabla_i W_r|\mathbf X)\,,\label{eq:nabla_W}
\end{align}
and
\begin{align}
	\nabla_i W_r%= & n^{-1/2} \left(J_{r,i}+D_{r,i}+\mathcal R_{r,i}\right)  \label{eq:nabla_W_JDR}\\
	 = & n^{-1/2} (\nabla_i g_{r,i} + D_{r,i}) \label{eq:nabla_W_gD}
\end{align}
with $D_{r,i}$, $g_{r,i}$ defined in \eqref{eq:D} and \eqref{eq:g}, respectively.
% \begin{align*}
% 	% \Delta_i Y_r &= \mathbb E(\nabla_i Y_r|F_i) = \frac{Y_{r,i}}{\sqrt{n}}  \\
% %	& W_i = \frac{1}{\sqrt{n}} \sum_{j \neq i} \left\{ \ell(X_j, \fvj(X^i)) - \E[ R(\fvj(X^i))] \right\} \\
% %	& Z_i(t) = \frac{1}{\sqrt{n}} \left( \sqrt{t} W_i + \sqrt{1 - t} Y_i \right) \\
% 	J_{r,i} & = \nabla_i \ell_r(X_i, \mathbf X_{-v_i})\,, \\
% 	\mathcal R_{r,i} & = \sum_{j \notin I_{v_i}} \nabla_i  R_r(\mathbf X_{-v_j})\,.
% 	% J_{r,i}(t) &= \sqrt{t} J_{r,i} / \sqrt{n} \\
% 	% R_{r,i}(t) &= \sqrt{t} \mathcal R_{r,i}  / \sqrt{n} \,,\\
% 	% Y_{r,i}(t) &= \sqrt{1-t} Y_{r,i} / \sqrt{n} \\
% 	% Z_r(t) - Z_r^i(t) & = J_{r,i}(t) + D_{r,i}(t) + \mathcal R_{r,i}(t) + Y_{r,i}(t) 
% \end{align*}
% By construction, $J_{r,i}$ is $O_P(1)$, $D_{r,i}$ is $O_P(\lstability)$ by \Cref{lem:D}, and $\mathcal R_{r,i}=O_P(\rstability)$ by \Cref{def:risk_stab}.

Similarly, for $t\in(0,1)$, consider interpolating variable
$$
Z_r(t) =\sqrt{t} W_r+\sqrt{1-t}Y_r\,,
$$
the leave-one-out version
$$
Z_r^i(t) = \sqrt{t} W_r^i +\sqrt{1-t} \sum_{j\neq i} \varepsilon_{r,j}/\sqrt{n}\,.
$$
and the martingale increment with respect to the filtration $\{F_i:1\le i\le n\}$,
$$
\Delta_i Z_r(t) = \sqrt{t}(\Delta_i W_r) + \sqrt{1-t}\varepsilon_{r,i}/\sqrt{n}\,.
$$

Consider $h:\mathbb R^p\mapsto \mathbb R$ with quantities $M_2$ and $M_3$ defined as in \eqref{eq:M2} and \eqref{eq:M3}.  In the following we will focus on controlling $\mathbb E\left[h(\mathbf W)-h(\mathbf Y)\right]$, where the bold font symbols represent the corresponding $p$-dimensional vectors: $\mathbf W=(W_1,...,W_p)$, $\mathbf Y=(Y_1,...,Y_p)$, and $\mathbf Z^i(t)=(Z_1^i(t),...,Z_p^i(t))$, etc. The only exception is $\mathbf X$, which corresponds to the collection of $n$ iid samples $(X_1,...,X_n)\in\mathcal X^n$.

Write $\frac{d \Delta_i Z_{r}(t)}{dt} = Z^\prime_{r,i}(t)$.
By Taylor expansion:
   \begin{align}
	\mathbb E\frac{dh(\mathbf Z(t))}{dt}& = \sum_{r=1}^p \sum_{i=1}^n \mathbb E[\partial_r h(\mathbf Z^i(t)) Z^\prime_{r,i}(t)] \nonumber\\
	&~~ + \sum_{s=1}^p \sum_{r=1}^p \sum_{i=1}^n \mathbb E[\partial_s \partial_r h(\mathbf Z^i(t)) \left( Z_s(t)-Z_s^i(t) \right) Z_{r,i}^\prime(t)]\nonumber \\
	& ~~ + \sum_{u=1}^p \sum_{s=1}^p \sum_{r=1}^p \sum_{i=1}^n \mathbb E\bigg\{ [Z_s(t)-Z_s^i(t)] [Z_u(t)-Z_u^i(t)]\nonumber\\ &~~~~\times\left[ \int_0^1 (1-v) \partial_u \partial_s \partial_r h(\mathbf Z^i(t) + v\mathbf Z_i(t)) dv \right] Z^\prime_{r,i}(t) \bigg\}.\label{eq:thm3_main_expansion}
\end{align}

The first term equals $0$ because $\mathbf Z^i(t)$ does not involve $(X_i,\varepsilon_i)$, and
$\mathbb E_{X_i,\varepsilon_i} Z_{r,i}'(t)=0$.

The second term can be written as
\begin{align}
	& \sum_{s=1}^p \sum_{r=1}^p \sum_{i=1}^n \mathbb E[\partial_s \partial_r h(\mathbf Z^i(t)) \left( Z_s(t)-Z_s^i(t) \right) Z_{r,i}^\prime(t) ] \nonumber\\
	 =& \frac{1}{2}\sum_{s=1}^p \sum_{r=1}^p \sum_{i=1}^n \mathbb E \left\{ \left(\frac{\Delta_i W_r}{\sqrt{t}} - \frac{\varepsilon_{r,i}}{\sqrt{n}\sqrt{1-t}} \right) \right.\nonumber\\
	 & \times \left. \left[ \sqrt{t}\left(W_{s} - W_{s}^i\right) + \sqrt{1-t}\frac{\varepsilon_{s,i}}{\sqrt{n}} \right] \partial_s \partial_r h(\mathbf Z^i(t))\right\} \nonumber\\
	% & ~~ = \sum_{s=1}^p \sum_{r=1}^p \sum_{i=1}^n \mathbb E \Big[ \left(\Delta_i W_{r,i}/\sqrt{t} - \Delta_i Y_{r,i}/\sqrt{1-t} \right) \\
	% & ~~~~ \times \left( \sqrt{t}\left(W_{s} - W_{s}^i\right) + \sqrt{1-t}\left(Y_{s} - Y_{s}^i\right)  \right) \partial_s \partial_r h(\mathbf Z^i(t)) \Big] \\
	=& \frac{1}{2} \sum_{s=1}^p \sum_{r=1}^p \sum_{i=1}^n \mathbb E \left\{ \left[ \Delta_i W_r \left( W_s - W_s^i \right) - \phi_{rs}/n \right] \partial_s \partial_r h(\mathbf Z^i(t)) \right\} \nonumber\\ 
	=& \frac{1}{2} \sum_{s=1}^p \sum_{r=1}^p \sum_{i=1}^n \mathbb E \left\{ \left[ \mathbb E(\nabla_i W_r | F_i) \mathbb E(\nabla_i W_s | \mathbf X) - \phi_{rs}/n \right] \partial_s \partial_r h(\mathbf Z^i(t)) \right\}\nonumber\\
	= & \frac{1}{2n} \sum_{s=1}^p \sum_{r=1}^p \sum_{i=1}^n \mathbb E \left\{ \left[ C_{rs,i} - \phi_{rs} + B_{rs,i} \right] \partial_s \partial_r h(\mathbf Z^i(t)) \right\}\,,\label{eq:thm3_q2}
	\end{align}
where $C_{rs,i}=\mathbb E(\nabla_i g_{r,i}|F_i)\mathbb E(\nabla_i g_{s,i}|\mathbf X)$ is the same as defined in \eqref{eq:C_rsi} and
\begin{align*}
B_{rs,i}=&\mathbb E(\nabla_i g_{r,i}|F_i)\mathbb E(D_{s,i}|\mathbf X)+\mathbb E(D_{r,i}|F_i)\mathbb E(\nabla_i g_{s,i}|\mathbf X)+\mathbb E(D_{r,i}|F_i)\mathbb E(D_{s,i}|\mathbf X)\,.
\end{align*}
In \eqref{eq:thm3_q2}, the first equation follows by construction of $Z_{r,i}'(t)$ and $Z_s^i$; the second equation follows by taking a conditional expectation over $\varepsilon_i$ and the definition of $\phi_{rs}$; the third and fourth equations follow from \eqref{eq:nabla_W} and \eqref{eq:nabla_W_gD}, respectively.

By \Cref{ass:lstability,ass:risk_stab,lem:D} we have $B_{rs,i}$ is $[(1+\rstability)\lstability]$-SW. Using the same argument as in \eqref{eq:SW_application_thm1q2}, we have 
\begin{align}\label{eq:thm3_q2_a}
	\frac{1}{2n}\mathbb E\sum_{rs,i}|B_{rs,i}\partial_s\partial_r h(\mathbf Z^i(t))|\le \tilde O\left[(1+\rstability)\lstability M_2 \right]\,.
\end{align}
Define $\bar C_{rs,i}=\mathbb E(C_{rs,i}|\mathbf X^{-i})$.
Now we are left with the term
\begin{align*}
&\frac{1}{2n} \sum_{s=1}^p \sum_{r=1}^p \sum_{i=1}^n \mathbb E \left\{ \left[ C_{rs,i} - \phi_{rs} \right] \partial_s \partial_r h(\mathbf Z^i(t)) \right\}\\
=&\frac{1}{2n} \sum_{s=1}^p \sum_{r=1}^p \sum_{i=1}^n \mathbb E \left\{ \left[ \bar C_{rs,i} - \phi_{rs} \right] \partial_s \partial_r h(\mathbf Z^i(t)) \right\}\\
=&\frac{1}{2n} \sum_{s=1}^p \sum_{r=1}^p \sum_{i=1}^n \mathbb E \left\{ \left[ \bar C_{rs,i}-\mathbb E C_{rs,i}+\mathbb E C_{rs,i} - \phi_{rs} \right] \partial_s \partial_r h(\mathbf Z^i(t)) \right\}
\end{align*}
where the equality holds by taking conditional expectation given $\mathbf X^{-i}$ and realizing $\mathbf Z^i(t)$ is independent of $X_i$.  By \Cref{lem:C_rsi}, $\bar C_{rs,i}-\mathbb E C_{rs,i}$ is $\lstability(1+\rstability)$-sub-Weibull. Repeating the truncation argument used in \eqref{eq:SW_application_thm1q2} we get
\begin{align}
	&\frac{1}{2n} \sum_{s=1}^p \sum_{r=1}^p \sum_{i=1}^n \mathbb E \left\{ \left[ \bar C_{rs,i}-\mathbb E C_{rs,i} \right] \partial_s \partial_r h(\mathbf Z^i(t)) \right\}\nonumber\\
	\le & \tilde O\left[\lstability(1+\rstability)M_2\right]\,.\label{eq:thm3_q2_b}
\end{align}
We still need to control $\mathbb E C_{rs,i}-\phi_{rs}$, which is provided by \Cref{lem:E(C_rsi)-phi_rs}. Thus we obtain
\begin{align}
	\frac{1}{2n} \sum_{s=1}^p \sum_{r=1}^p \sum_{i=1}^n \mathbb E \left\{ |\mathbb E C_{rs,i} - \phi_{rs}| \times |\partial_s \partial_r h(\mathbf Z^i(t))| \right\}\le c \lstability M_2\,,\label{eq:thm3_q2_c}
\end{align}
for some universal constant $c$.

Combining \eqref{eq:thm3_q2_a}, \eqref{eq:thm3_q2_b}, and \eqref{eq:thm3_q2_c} into \eqref{eq:thm3_q2} we conclude that the second term in \eqref{eq:thm3_main_expansion} is upper bounded in absolute value by (using the simplifying assumption $\lstability<1$.)
\begin{equation}\label{eq:thm3_q2_bound}
\tilde O\left[\lstability(\rstability+1)M_2\right]\,.
\end{equation}

The third term in \eqref{eq:thm3_main_expansion} can be controlled using the following equation
\begin{align}
Z_r(t)-Z_r^i(t) = & \frac{1}{\sqrt{n}}\mathbb E\left[g_{r,i}\sqrt{t}+D_{r,i}\sqrt{t}+\varepsilon_{r,i}\sqrt{1-t}\Big|\mathbf X\right]\,,\label{eq:nabla_Z(t)}
\end{align}
which holds by combining \eqref{eq:nabla_W} and \eqref{eq:nabla_W_gD}, and, by \Cref{ass:lstability,ass:risk_stab,lem:D}, is
 $n^{-1/2}(1+\rstability)$-sub-Weibull. 

Similarly, by \eqref{eq:delta2nabla} and \eqref{eq:nabla_W_gD} we have
$$
Z_{r,i}'(t)=\frac{1}{2\sqrt{n}}\left(\frac{\mathbb E(\nabla_i g_{r,i}+D_{r,i}|F_i)}{\sqrt{t}}-\frac{\varepsilon_{r,i}}{\sqrt{1-t}}\right)\,,
$$
and hence
$$
\frac{Z_{r,i}'(t)}{n^{-1/2}\eta_t(1+\rstability)}
$$
is sub-Weibull, where $\eta_t$ is defined in \eqref{eq:eta_t}.

Putting together the sub-Weibull properties of $Z_r(t)-Z_r^i(t)$ and $Z_{r,i}'(t)$, we have
$$
\frac{[Z_s(t)-Z_s^i(t)][Z_u(t)-Z_u^i(t)]Z_{r,i}'(t)}{n^{-3/2}\eta_t(1+\rstability)^3}
$$
is sub-Weibull.
Therefore, applying the truncation argument in \eqref{eq:SW_application_thm1q2} again in the third term of \eqref{eq:thm3_main_expansion}, we obtain an upper bound of
\begin{align}
	\tilde O\left\{n^{-1/2}\eta_t(1+\rstability)^3 M_3\right\}\,.\label{eq:thm3_q3_bound}
\end{align}
Combining \eqref{eq:thm3_q2_bound} and \eqref{eq:thm3_q3_bound} with \eqref{eq:thm3_main_expansion} and integrate the latter over $t\in(0,1)$ we obtain
\begin{align}
	|\mathbb E[h(\mathbf W)-h(\mathbf Y)]|\le \tilde O\left[\lstability(1+\rstability)M_2+n^{-1/2}(1+\rstability)^3 M_3\right]\,.\label{eq:thm3_h_bound}
\end{align}

Again, using \Cref{lem:h_function} and the anti-concentration result \citep[][Lemma 2.1]{chernozhukov2013gaussian}, we have for any $\beta>0$
\begin{align}
& \sup_{z}\left|\mathbb P\left(\max_r W_r \le z\right)-\mathbb P(\max_r Y_r\le z)\right|\\
& \le \tilde O\left[\lstability(1+\rstability)\beta^2+n^{-1/2}(1+\rstability)^3 \beta^3 + \beta^{-1}\right]\nonumber\\
& \le \tilde O \left([\lstability(1+\rstability)]^{1/3}\vee n^{-1/8}(1+\rstability)^{3/4}\right) \,\label{eq:thm3_max_comp1}
\end{align}
where the last inequality follows by choosing $\beta=\min([\lstability(1+\rstability)]^{-1/3}, n^{1/8}(1+\rstability)^{-3/4})$.

\end{proof}

\subsection{Proof of variance estimation with deterministic centering (\Cref{thm:marginal_var})}
% Recall that $$\hat R_{{\rm cv},r}=\hat R_{{\rm cv},r}(\mathbf X)=\frac{1}{n}\sum_{i=1}^n \ell(X_i;\mathbf X_{-v_i})\,.$$  

The claimed result in \Cref{thm:marginal_var} follows directly from the two following lemmas.
\begin{lemma}\label{thm:asymp_var_nabla_i}
Under \Cref{ass:condition_l,ass:lstability,ass:risk_stab,ass:phi_bound},	
	$$
\sup_{1\le r,s\le p}\left|\frac{n^2}{2}\mathbb E \left(\nabla_i \hat R_{{\rm cv},r}\nabla_i\hat R_{{\rm cv},s}\right) - \phi_{rs}\right|=O\left(\lstability(1+\rstability)\right)\,.
	$$
\end{lemma}
\Cref{thm:asymp_var_nabla_i} shows that the variance of $\nabla_i \hat{\mathbf R}_{{\rm cv}}$ is entry-wise close to the true covariance matrix $\Phi$.  The next result further reduces this variance term to the proportion contributed solely by $X_i$.

\begin{lemma}\label{thm:concentration_nabla_i_f}
	Under \Cref{ass:condition_l,ass:lstability,ass:risk_stab,ass:phi_bound},	
	$$n^2\sup_{1\le r,s\le p}\left|\mathbb E\left[\nabla_i \hat R_{{\rm cv},r}\nabla_i\hat R_{{\rm cv},s}|\mathbf X_{-v_i}\right]-
\mathbb E\left[\nabla_i \hat R_{{\rm cv},r}\nabla_i\hat R_{{\rm cv},s}\right]
	\right|\lesssim \tilde O\left[\lstability(1+\rstability)\right]\,.$$
\end{lemma}

\begin{proof}[Proof of \Cref{thm:asymp_var_nabla_i}]
	Let $f_r=n\hat R_{{\rm cv},r}=\sum_{i=1}^n \ell_r(X_i;\mathbf X_{-v_i})$.
	For $-1\le j\le n$, $j\neq i$, define 
	\begin{align*}
		E_{j,i} = \left\{\begin{array}{ll}
			F_0\,, & j=-1\\
			\sigma(X_i,X_i')\,, & j=0\\
			\sigma(X_1,...,X_j,X_i,X_i')\,, & 1\le j\le n\,.
		\end{array}\right.
	\end{align*}
	and
	\begin{align*}
		E_{j,i}^-=\left\{\begin{array}{ll}
			E_{j-1,i}\,, & j\neq i+1\,,\\
			E_{i-1,i}\,, & j= i+1\,.
		\end{array}\right.
	\end{align*}
Then $\{E_{j,i}:-1\le j\le n,~j\neq i\}$ is  a filtration.

Use notation
$\ell_{r,i}=\ell_r(X_i;\mathbf X_{-v_i})$, $\bar\ell_{r,i}=\bar\ell_r(X_i)$, $\ell_{r,i}'=\ell_r(X_i';\mathbf X_{-v_i})$, $\bar\ell_{r,i}'=\bar\ell_r(X_i')$, and $\bar R_{r,i}=\bar R_r(X_i)$.

Using the decomposition
$$
\nabla_i f_r = \nabla_i K_{r,i} + \nabla_i \mathcal R_{r,i} + D_{r,i}
$$
we get
\begin{align}
\nabla_i f_r\nabla_i f_s = &\nabla_i K_{r,i} \nabla_i K_{s,i} +\nabla_i \mathcal R_{r,i}\nabla_i \mathcal R_{s,i}+\nabla_i K_{r,i}\nabla_i \mathcal R_{s,i}+\nabla_i \mathcal R_{r,i}\nabla_i K_{s,i}\nonumber\\
&+D_{r,i}\nabla_i K_{s,i}+D_{s,i}\nabla_i K_{r,i}+D_{r,i}\nabla_i \mathcal R_{s,i}+D_{s,i}\nabla_i\mathcal R_{r,i}+D_{r,i}D_{s,i}\,.\label{eq:thm3_var_decomp}
\end{align}

For the first term
\begin{align*}
	&\mathbb E \nabla_i K_{r,i} \nabla_i K_{s,i}\\ = & \mathbb E \nabla_i \ell_{r,i}\nabla_i\ell_{s,i}\\
 =& {\rm Cov}\left\{\mathbb E\left[\nabla_i \ell_{r,i}|X_i,X_i'\right],~\mathbb E\left[\nabla_i \ell_{s,i}|X_i,X_i'\right]\right\}\\
 & + \mathbb E\left\{ {\rm Cov}\left[\nabla_i \ell_{r,i},~\nabla_i \ell_{s,i}|X_i,X_i'\right] \right\}\\
 = & 2 {\rm Cov}(\bar\ell_{r,i},\bar\ell_{s,i})+\mathbb E\left\{ {\rm Cov}\left[\nabla_i \ell_{r,i},~\nabla_i \ell_{s,i}|X_i,X_i'\right] \right\}
\end{align*}
where the second term is upper bounded by
\begin{align*}
	&\left|\mathbb E\left\{ {\rm Cov}\left[\nabla_i \ell_{r,i},~\nabla_i \ell_{s,i}|X_i,X_i'\right] \right\}\right|\\
	\le &\frac{1}{2}\mathbb E\left\{{\rm Var}\left[\nabla_i\ell_{r,i}|X_i,X_i'\right]+{\rm Var}\left[\nabla_i\ell_{s,i}|X_i,X_i'\right]\right\}\\
	\le & 2 \mathbb E\left\{{\rm Var}(\ell_{r,i}|X_i)+{\rm Var}(\ell_{s,i}|x_i)\right\}\\
	\le & \sum_{j\neq I_{v_i}}\mathbb E(\nabla_j \ell_{r,i})^2+\mathbb E(\nabla_j \ell_{s,i})^2\\
	\le & 2\lstability^2
\end{align*}
where the second last step used Efron-Stein inequality.

So we conclude
\begin{equation}\label{eq:thm3_var_term1}
	\left|\mathbb E \nabla_i K_{r,i}\nabla_i K_{s,i} - 2 {\rm Cov}(\bar\ell_{r,1},\bar\ell_{s,1})\right|\le 2\lstability^2
\end{equation}

For the second term in \eqref{eq:thm3_var_decomp},
using the martingale decomposition
\begin{align*}
	\nabla_i \mathcal R_{r,i}=\sum_{0\le j\le n,j\neq i} \mathbb E(\nabla_i \mathcal R_{r,i}|E_{j,i})-\mathbb E(\nabla_i\mathcal R_{r,i}|E_{j,i}^-)
\end{align*}
 we have, by orthogonality between martingale increments,
\begin{align*}
&\mathbb E\nabla_i \mathcal R_{r,i} \nabla_i\mathcal R_{s,i}\\
 = & \mathbb E\sum_{0\le j\le n,j\neq i}\left[\mathbb E(\nabla_i\mathcal R_{r,i}|E_{j,i})-\mathbb E(\nabla_i\mathcal R_{r,i}|E_{j,i}^-)\right]\left[\mathbb E(\nabla_i\mathcal R_{s,i}|E_{j,i})-\mathbb E(\nabla_i\mathcal R_{s,i}|E_{j,i}^-)\right]\\
 = & 2\tilde n^2 {\rm Cov}(\bar R_r(X_1),\bar R_s(X_1))+\mathbb E\sum_{1\le j\le n,j\neq i}\left[\mathbb E(\nabla_j\nabla_i\mathcal R_{r,i}|E_{j,i})\mathbb E(\nabla_j\nabla_i\mathcal R_{s,i}|E_{j,i})\right]
\end{align*}
and the remainder term satisfies
\begin{align*}
	&\left|\mathbb E\sum_{1\le j\le n,j\neq i}\left[\mathbb E(\nabla_j\nabla_i\mathcal R_{r,i}|E_{j,i})\mathbb E(\nabla_j\nabla_i\mathcal R_{s,i}|E_{j,i})\right]\right|\\
	\le &\frac{1}{2}\sum_{1\le j\le n, j\neq i} \left[\|\nabla_j\nabla_i \mathcal R_{r,i}\|_2^2+\|\nabla_j\nabla_i \mathcal R_{s,i}\|_2^2\right]\\
	\le & \lstability^2
\end{align*}
Hence we have
\begin{equation}\label{eq:thm3:var_term2}
	\left|\mathbb E\nabla_i \mathcal R_{r,i} \nabla_i\mathcal R_{s,i}-2\tilde n^2 {\rm Cov}(\bar R_{r,1},\bar R_{s,1})\right|\le \lstability^2\,.
\end{equation}

For the third term in \eqref{eq:thm3_var_decomp},
using the martingale decomposition of $\nabla_i K_{r,i}$ and $\nabla_i \mathcal R_{s,i}$, we have
\begin{align*}
 &	\mathbb E \nabla_i K_{r,i}\nabla_i \mathcal R_{s,i}\\
=&  \mathbb E \left[\mathbb E(\nabla_i K_{r,i}|X_i,X_i') \mathbb E(\nabla_i\mathcal R_{s,i}|X_i,X_i')\right]
+\sum_{1\le j\le n,j\neq i}\mathbb E\left[\mathbb E(\nabla_j\nabla_i K_{r,i})\mathbb E(\nabla_j\nabla_i \mathcal R_{s,i})\right]\\
=&2\tilde n {\rm Cov}(\bar\ell_{r,1},\bar R_{s,1})+\sum_{1\le j\le n,j\neq i}\mathbb E\left[\mathbb E(\nabla_j\nabla_i K_{r,i})\mathbb E(\nabla_j\nabla_i \mathcal R_{s,i})\right]
\end{align*}
where the remainder term satisfies
\begin{align*}
	&\left|\sum_{1\le j\le n,j\neq i}\mathbb E\left[\mathbb E(\nabla_j\nabla_i K_{r,i})\mathbb E(\nabla_j\nabla_i \mathcal R_{s,i})\right]\right|\\
	\le & \sum_{1\le j\le n,j\neq i}\|\nabla_j\nabla_i K_{r,i}\|_2\|\nabla_j\nabla_i\mathcal R_{s,i}\|_2\\
	\le & \lstability^2\,.
\end{align*}
The fourth term can be bounded similarly. So we have
\begin{align}
\left|\nabla_i K_{r,i}\nabla_i \mathcal R_{s,i}+\nabla_i \mathcal R_{r,i}\nabla_i K_{s,i}-2\tilde n {\rm Cov}(\bar\ell_{r,1}\bar R_{s,1})-2\tilde n{\rm Cov}(\bar R_{r,1},\bar\ell_{s,1})\right|\le 2\lstability^2\,.
\end{align}
For the other terms in \eqref{eq:thm3_var_decomp}, according to \Cref{lem:D}, \Cref{ass:lstability}, and \Cref{ass:risk_stab}, we have $\|D_{r,i}\|_2\lesssim \lstability$, $\|\nabla_i K_{r,i}\|_2\lesssim 1$, and $\nabla_i \mathcal R_{r,i}\le \rstability$, so that
the last five terms in \eqref{eq:thm3_var_decomp} are bounded by, up to constant factor,
$\lstability(1+\rstability)$.  The claimed result is proved.
\end{proof}

\begin{proof}[Proof of \Cref{thm:concentration_nabla_i_f}]
We follow the notation in the proof of \Cref{thm:asymp_var_nabla_i}.	By the symmetry assumption of $\ell$, we can assume $i>\tilde n$ without loss of generality.

	Let $M=\nabla_i f_r\nabla_i f_s-\phi_{rs}$, and $M_j=\mathbb E(M|F_{j})-\mathbb E(M|F_{j-1})=\mathbb E(\nabla_j M|F_j)$ for $j=1,...,\tilde n$.  The main task in the proof is to control $\|M_j\|_{\psi_\alpha}$, which further reduces to controlling the norm of $\nabla_j(\nabla_i f_r\nabla_i f_s)$.

	To begin with, we first write
	$$
\nabla_j(\nabla_i f_r\nabla_i f_s) = \nabla_j\nabla_i f_r \nabla_i f_s +\nabla_i f_s(\mathbf X^j)\nabla_j\nabla_i f_r\,.
	$$

Use the decomposition $\nabla_i f_r= \nabla_i g_{r,i}+D_{r,i}$
we have
\begin{align*}
&\left\|\mathbb E\left(\nabla_i f_r\nabla_i f_s|\mathbf X_{-v_i}\right)-\mathbb E (\nabla_i f_r\nabla_i f_s)\right\|\\
\le&\left\|\mathbb E\left[\nabla_i g_{r,i}\nabla_i g_{s,i}|\mathbf X_{-v_i}\right]-\mathbb E(\nabla_i g_{r,i}\nabla_i g_{s,i})\right\|\\
&+\left\|\mathbb E\left[\nabla_i g_{r,i}D_{s,i}|\mathbf X_{-v_i}\right]-\mathbb E(\nabla_i g_{r,i}D_{s,i})\right\|\\
&+\left\|\mathbb E\left[D_{r,i}\nabla_i g_{s,i}|\mathbf X_{-v_i}\right]-\mathbb E(D_{r,i}\nabla_i g_{s,i})\right\|\\
&+\left\|\mathbb E\left[D_{r,i}D_{s,i}|\mathbf X_{-v_i}\right]-\mathbb E(D_{r,i}D_{s,i})\right\|\,.
\end{align*}
Using the fact that for any random variable $W$
$$
\|W-\mathbb E W\|_q\le 2\|W\|_q\,,~~\forall~~q\ge 1\,,
$$
by assumption $D_{r,i}$ is $\lstability$-SW, $\nabla_i g_{r,i}$ is $(1+\rstability)$-SW, so the sum of the last three terms in the above expression is
$\lstability(1+\rstability)$-SW.

Now for the first term,
Let $M=\mathbb E\left[\nabla_i g_{r,i}\nabla_i g_{s,i}|\mathbf X_{-v_i}\right]-\mathbb E(\nabla_i g_{r,i}\nabla_i g_{s,i})$ and $M_j=\mathbb E(M|F_j)-\mathbb E(M|F_{j-1})=\mathbb E(\nabla_j M|F_j)$ for $1\le j\le \tilde n$. The main remaining task in the proof is to control the tail of $M_j$.

For any $\ell_q$ norm $\|\cdot\|$ with $q\ge 1$,
\begin{align*}
	&\|M_j\|=\|\mathbb E(\nabla_j M|F_j)\|\le \|\nabla_j M\|\\
	=&\left\|\nabla_j\left[\mathbb E(\nabla_i g_{r,i}\nabla_i g_{s,i})|\mathbf X_{-v_i}\right]\right\|\\
	\le & \|\nabla_j(\nabla_ig_{s,i}\nabla_ig_{r,i})\|\\
	= & \left\|\nabla_j\nabla_j g_{s,i}\nabla_i g_{s,i}+\nabla_i g_{s,i}(\mathbf X^j)\nabla_j\nabla_ig_{s,i}\right\|\\
	\le & 2n^{-1/2}\lstability(1+\rstability)\,.
\end{align*}
Then using \Cref{lem:MG-subW} we conclude $M$ is $\lstability(1+\rstability)$-SW\,.
\end{proof}

\subsection{Auxiliary lemmas}
\begin{lemma}[Bounding $\bar C_{rs,i}-\mathbb E C_{rs,i}$]\label{lem:C_rsi}
Under the conditions in \Cref{thm:thm3},
	$\bar C_{rs,i}-\mathbb E C_{rs,i}$ is $\lstability(1+\rstability)$-SW.
\end{lemma}
\begin{proof}[Proof of \Cref{lem:C_rsi}]
% 	By Efron-Stein inequality
% \begin{align}
% 	\mathbb E\left[\bar C_{rs,i}-\mathbb E C_{rs,i}\right]^2\le & 2 \sum_{j\neq i}\|\nabla_j \bar C_{rs,i}\|_2^2
% 	\label{eq:efron-stein_C_rsi}
% \end{align}

For $j\neq i$, and function $f$ acting on $\mathbf X$, let $\mathbf X^{j,-i}$ be the vector obtained by replacing $X_j$ in $\mathbf X^{-i}$ with its iid copy $X_j'$.
Then by Jensen's inequality, we have, for $q\ge 1$
\begin{align}
\|\nabla_j \mathbb E(f|\mathbf X^{-i})\|_q^q=&\mathbb E\left\{\mathbb E(f|\mathbf X^{-i})-\mathbb E \left[f(\mathbf X^{j,-i})|\mathbf X^{j,-i}\right]\right\}^q\nonumber\\
=&\mathbb E\left\{\mathbb E\left[f(\mathbf X)-f(\mathbf X^{j})|\mathbf X^{-i},X_j'\right]\right\}^q\nonumber\\
\le & \|\nabla_j f\|_q^q\,.\label{eq:monotone_nabla_conditional}
\end{align}
Take $f$ to be $C_{rs,i}$, we have for $j\neq i$.
$$
\|\nabla_j \bar C_{rs,i}\|_q\le \|\nabla_j C_{rs,i}\|_q\,.
$$
Next we control $\|\nabla_j C_{rs,i}\|_q$. By definition,
\begin{align}
	\nabla_j C_{rs,i}=&\nabla_j\left[\mathbb E(\nabla_i g_{r,i}|F_i)\mathbb E(\nabla_i g_{s,i}|\mathbf X)\right]\nonumber\\
=&\left[\nabla_j\mathbb E(\nabla_i g_{r,i} | F_i) \right] \mathbb E(\nabla_i g_{s,i} | \mathbf X) + \left[\nabla_j\mathbb E(\nabla_i g_{s,i} | \mathbf X) \right] \left\{\mathbb E\left[\nabla_i g_{r,i}(\mathbf X^j) | F_i(\mathbf X^j)\right]\right\}\,.\label{eq:nabla_C_expand}
\end{align}
% Now using Cauchy-Schwartz and H\"{o}lder in \eqref{eq:nabla_C_expand},
% \begin{align}
% 	\|\nabla_j C_{rs,i}\|_2^2\lesssim \|\nabla_j\nabla_i g_{r,i}\|_4^2\|\nabla_i g_{s,i}\|_4^2+\|\nabla_j\nabla_i g_{s,i}\|_4^2\|\nabla_i g_{r,i}\|_4^2\,.\label{eq:nabla_C_bound_1}
% \end{align}
By \Cref{pro:basic_SW} and definition of $\lstability$ and $\rstability$,
\begin{align}
	\nabla_j\nabla_i g_{r,i} =\nabla_j\nabla_i \ell_r(X_i;\mathbf X_{-v_i})+\sum_{k\notin I_{v_i}}\nabla_j\nabla_i R(\mathbf X_{-v_k}) ~~\text{is}~~ n^{-1/2}\lstability\text{-SW}\,,\label{eq:nabla_ji_g}
\end{align}
and
\begin{equation}\label{eq:nabla_i_gi_sw}
\nabla_i g_{r,i}=\nabla_i \ell_r(X_i;\mathbf X_{-v_i})+\sum_{k\notin I_{v_i}}\nabla_i R(\mathbf X_{-v_k}) ~~\text{is}~~ (1+\rstability)\text{-SW}\,.
\end{equation}
Combining \eqref{eq:nabla_ji_g} and \eqref{eq:nabla_i_gi_sw} with \eqref{eq:nabla_C_expand}, and applying \Cref{pro:basic_SW} we conclude that $\nabla_j C_{rs,i}$ is $n^{-1/2}\lstability(1+\rstability)$-SW.
For the same reason as \eqref{eq:monotone_nabla_conditional}, we have $\nabla_j \bar C_{rs,i}$ is $n^{-1/2}\lstability(1+\rstability)$-SW.
Then the desired result follows from applying \Cref{lem:MG-subW} to the martingale sequence obtained by taking conditional expectation of $\bar C_{rs,i}$ with respect to the filtration $(F_{j,i}:1\le 0\le n,~j\neq i)$, where $F_{j,-i}$ is the sigma field generated by $(X_k:1\le k\le j,~k\neq i)$.
\end{proof}

\begin{lemma}[Bounding $\mathbb E C_{rs,i}-\phi_{rs}$]\label{lem:E(C_rsi)-phi_rs}
There exists a universal constant $c>0$ such that for all $(r,s,i)\in[p]^2\times [n]$, $|\mathbb E C_{rs,i}-\phi_{rs}|\le c \lstability$.
\end{lemma}
\begin{proof}[Proof of \Cref{lem:E(C_rsi)-phi_rs}]
Since $F_i$ is a sub $\sigma$-field of $\mathbf X$, and $\nabla_j g_{r,i}$ is centered, we have
	\begin{align}\label{eq:E(C_rsi)_expand_1}
		\mathbb E C_{rs,i} = \mathbb E\left[\mathbb E(\nabla_i g_{r,i}|F_i)\mathbb E(\nabla_i g_{s,i}|F_i)\right] 
	\end{align}
	For $j\in\{1,...,i-1\}$, let $H_{j,i}$ be the $\sigma$-field generated by $(X_1,...,X_j,X_i)$.  Define $H_{0,i}$ as the $\sigma$-field generated by $X_i$, and $H_{-1,i}$ be the trivial $\sigma$-field.
	Then we can write the martingale decomposition of $\mathbb E(\nabla_i g_{r,i}|F_i)$ as follows
	\begin{align*}
		\mathbb E (\nabla_i g_{r,i}|F_i) = \sum_{j=0}^{i-1} \mathbb E(\nabla_i g_{r,i}|H_j)-\mathbb E(\nabla_i g_{r,i}|H_{j-1})\,.
	\end{align*}
	Apply this decomposition to both $\mathbb E (\nabla_i g_{r,i}|F_i)$ and $\mathbb E (\nabla_i g_{s,i}|F_i)$ in \eqref{eq:E(C_rsi)_expand_1}, we get
	\begin{align}
		\mathbb E C_{rs,i} =& \mathbb E\sum_{j,k=0}^{i-1}\left[\mathbb E(\nabla_i g_{r,i}|H_j)-\mathbb E(\nabla_i g_{r,i}|H_{j-1})\right]\left[\mathbb E(\nabla_i g_{s,i}|H_k)-\mathbb E(\nabla_i g_{s,i}|H_{k-1})\right]\nonumber\\
		=&\mathbb E\sum_{j=0}^{i-1}\left[\mathbb E(\nabla_i g_{r,i}|H_j)-\mathbb E(\nabla_i g_{r,i}|H_{j-1})\right]\left[\mathbb E(\nabla_i g_{s,i}|H_j)-\mathbb E(\nabla_i g_{s,i}|H_{j-1})\right]\,,\label{eq:E(C_rsi)_expand_2}
	\end{align}
where the second equality holds because $H_0\subset H_1\subset H_2\subset\cdots\subset H_{i-1}$ is a filtration.

	When $j=0$, 
	\begin{align}
&\mathbb E\left\{\left[\mathbb E(\nabla_i g_{r,i}|H_j)-\mathbb E(\nabla_i g_{r,i}|H_{j-1})\right]\left[\mathbb E(\nabla_i g_{s,i}|H_j)-\mathbb E(\nabla_i g_{s,i}|H_{j-1})\right] \right\}\nonumber\\
=&\mathbb E\left\{ \left[\mathbb E (g_{r,i}|X_i) - \mathbb E g_{r,i}\right]
\left[\mathbb E (g_{s,i}|X_i) - \mathbb E g_{s,i}\right]\right\}\nonumber\\
=& \phi_{rs}\,.	\label{eq:E(C_rsi)_main}\end{align}
When $j>0$, we have, using Jensen's inequality and \eqref{eq:nabla_ji_g},
\begin{align}
	&\|\mathbb E(\nabla_i g_{r,i}|H_j)-\mathbb E(\nabla_i g_{r,i}|H_{j-1})\|_2=\|\mathbb E(\nabla_j\nabla_i g_{r,i}|H_j)\|_2\nonumber\\
	\le & \|\nabla_j\nabla_i g_{r,i}\|_2\lesssim n^{-1/2}\lstability\,.\label{eq:E(C_rsi)_remainder}
\end{align}
The claim follows by combining \eqref{eq:E(C_rsi)_main}, \eqref{eq:E(C_rsi)_remainder} with \eqref{eq:E(C_rsi)_expand_2}.
\end{proof}

\bibliographystyle{plainnat}
\bibliography{ref}
\end{document}